\documentclass[11pt]{article}
\usepackage[utf8]{inputenc}
\usepackage[english]{babel}

\usepackage{amsfonts}
\usepackage{amsmath,amsthm,amscd,amssymb,mathrsfs,setspace,comment}
\usepackage{mathtools}
\usepackage{cite}
\usepackage{latexsym,epsf,epsfig}
\usepackage{color}
\usepackage[a4paper,top=1in,bottom=1in,left=1in,right=1in,marginparwidth=0.5in]{geometry}
\usepackage[colorlinks,citecolor=lccx]{hyperref}

\usepackage[nameinlink,capitalise]{cleveref}
\usepackage[mathscr]{euscript}
\usepackage{dutchcal}
\usepackage{algorithm}
\usepackage{algorithmicx}
\usepackage{algpseudocode}
\usepackage[table]{xcolor}
\usepackage{soul}
\usepackage{subcaption}
\usepackage{siunitx}
\usepackage{tikz}
\usetikzlibrary{decorations.pathreplacing,calligraphy}

\usepackage{pgfplots}

\usepackage{booktabs}
\usepackage{multirow}
\usepackage{adjustbox}
\usepackage{authblk}
\usepackage{enumitem}

\newcommand{\argmin}{\arg\min}
\renewcommand{\st}{\operatorname{s.t.}}
\newcommand{\conv}{\operatorname{conv}}
\newcommand{\mcc}{\operatorname{mcc}}
\newcommand{\BV}{\operatorname{BV}}
\newcommand{\TV}{\operatorname{TV}}

\newcommand{\N}{\mathbb{N}}
\newcommand{\R}{\mathbb{R}}

\newcommand{\calQ}{\mathcal{Q}}
\newcommand{\Wad}{W^{\text{ad}}}
\newcommand{\Wadh}{W^{\text{ad}}_h}
\newcommand{\Zad}{Z^{\text{ad}}}
\newcommand{\Zadtau}{Z^{\text{ad}}_\tau}
\newcommand{\weakto}{\rightharpoonup}

\newcommand{\ydel}{y^\delta}

\definecolor{lccx}{HTML}{92268F}
\definecolor{ceruleanblue}{rgb}{0.16, 0.32, 0.75}

\ifpdf
\hypersetup{
	linkcolor=lccx
}
\fi

\title{Locally-averaged McCormick Relaxations for Discretization-regularized Inverse Problems}
\date{\today}

\definecolor{darkgreen}{rgb}{0,0.5,0}

\newcommand{\revision}[1]{{ #1}}
\newcommand{\revisionown}[1]{{ #1}}
\newcommand{\Margin}[1]{}

\newtheorem{theorem}{Theorem}[section]
\newtheorem{proposition}[theorem]{Proposition}

\theoremstyle{definition}

\newtheorem{assumption}[theorem]{Assumption}

\theoremstyle{remark}
\newtheorem{remark}[theorem]{Remark}

\crefname{assumption}{Assumption}{Assumptions}
\Crefname{assumption}{Assumption}{Assumptions}

\author[1]{Barbara Kaltenbacher}
\author[2]{Paul Manns}
\affil[1]{Department of Mathematics, Alpen-Adria-Universit\"{a}t, Klagenfurt, Austria, \textit{barbara.kaltenbacher@aau.at}}
\affil[2]{Department of Mathematics, TU Dortmund University, Dortmund, Germany, \textit{paul.manns@tu-dortmund.de}}

\begin{document}
\maketitle

\begin{abstract}
In this paper, by means of a standard model problem, we devise an approach to computing
approximate dual bounds for use in global optimization of coefficient identification
in partial differential equations (PDEs) by, e.g., (spatial) branch-and-bound methods.
Linearization is achieved by a McCormick relaxation (that is, replacing the bilinear PDE term 
by a linear one and adding inequality constraints), combined with local averaging to reduce the
number of inequalities. Optimization-based bound tightening allows us to tighten the relaxation and thus reduce the induced error. Combining this with a quantification of the discretization error and the propagated noise,
we prove that the resulting discretization regularizes the inverse problem, thus leading to an overall convergent scheme.
Numerical experiments illustrate the theoretical findings.
\end{abstract}
\section{Introduction}
The identification of coefficients in partial differential equations (PDEs) from indirect observations is a task that arises in a multitude of applications ranging from quantitative medical imaging via geophysical exploration to material characterization and its formulation as an operator equation is inherently nonlinear and ill-posed. Its regularized solution via variational methods, cf., e.g., \cite{EnglHankeNeubauer:1996,Groe84,Kirsch2021,SGGHL08,TikArs77}, leads to a nonconvex minimization problem and convergence of the regularized reconstruction to the true coefficient
requires computation of a global minimizer --- a task that has hardly been addressed in the literature so far, cf.  
\cite{Harrach_23,HarrachOW_2024,Wang04032015} and the references therein.
The aim of this paper is to contribute to the latter by devising a computationally tractable approach to determining dual bounds, that is,
lower bounds in the case of minimization. Dual bounds are required for global optimization algorithms to certify global optimality
and their computation can yield candidates for initializing gradient-based local optimization algorithms close to (approximate) global minimizers.
To this end, we employ a McCormick relaxation, combined with local averaging on a partition of the PDE domain and the induced discretization suggests itself for being used for regularization purposes as well. By means of a standard model problem---recovery of a potential in an elliptic PDE
with a bilinear term---we provide a framework that allows us to quantify both the approximation error and the propagated observational noise,
thus leading to an analysis of the resulting scheme in the sense of regularization theory.

\medskip

Let $\Omega \subset \R^d$, $d \in \N$, be a bounded, Lipschitz domain. We are interested in solving a prototypical inverse problem on $\Omega$, where the goal is to recover a function $w^\dagger$ that is a priori known to abide to bound constraints. The function $w^\dagger$ cannot be observed directly but only through indirect measurements $\revision{\ydel\approx y}=Tu$ of a state $u$,
which solves an elliptic PDE of the form $Au + f_0 uw = f_1$ that is governed by a bounded, linear, and continuously invertible operator
$A : U \to L^2(\Omega)$ (that is, an isomorphism), for some state space $U \hookrightarrow L^2(\Omega)$ and features the nonnegative multiplier
$f_0\in L^\infty(\Omega)$ and source term $f_1 \in L^2(\Omega)$ \footnote{These requirements can be weakened for some of the results in this paper
and need to be slightly strengthened for some others.} for the input $w = w^\dagger$. Formulating the \revision{noiseless} identification problem as an optimization problem yields
\Margin{R2 1.}
\begin{gather}\label{eq:inv}
\begin{aligned}
\min_{u,w}\ & \|Tu - \revision{y}\|_{Y}  \\
\st\    & Au + f_0 uw = f_1, \\
        & w \in \Wad 
        \subseteq
        \{
        w \in L^\infty(\Omega) \, : \, w_\ell(x) \le w(x) 
        \le w_u(x) \text{ for a.e. } x \in \Omega\},
\end{aligned}\tag{I}
\end{gather}
where $w_\ell,\,w_u\in L^\infty(\Omega)$ are given functions. This comprises the possible choice $\Wad\coloneqq \{w \in L^\infty(\Omega) : w(x) \in W \text{ for a.e. } x \in \Omega\}$ for a finite set $W\subseteq\mathbb{R}$ of a priori known parameter values. 
In our domain of interest, the optimal objective value is zero for the feasible and globally optimal pair $(u^\dagger, w^\dagger)$.
While this is a common property of least squares formulations of inverse problems, it will not be needed for our analysis in the remainder,
thus allowing for a perspective on formulations with alternative objective functions.

In \eqref{eq:inv}, $T : L^2(\Omega)\to Y$ is some observation operator and $y^\dagger \in Y$ the exact observational data. The most simple case of
$T$ is the identity with $Y=L^2(\Omega)$; realistically for imaging applications, $T$ is the trace on a boundary part or on a small open subset
$\omega$ of $\Omega$, thus $Y=L^2(\omega)$.

We briefly note at this point that when applying gradient-based solvers to the problems we derive from \eqref{eq:inv}, we do not change
the solution set of \eqref{eq:inv} by squaring the objective, which then gives a differentiable objective.

The available measurements of $u^\dagger$ are not exact but distorted by noise with noise level $\delta > 0$ 
\begin{equation}\label{eq:delta}
\|y^\dagger -\ydel\|_Y\leq\delta,
\end{equation}
which necessitates the introduction of a regularization technique. A well-established technique
is \emph{regularization by discretization} \cite{EnglHankeNeubauer:1996,GroetschNeubauer:1988,Kirsch2021,VainikkoHaemarik:1985},
that is, the variable $w$ is restricted to a finite-dimensional ansatz space $V_h$ with underlying mesh size $h > 0$.
Specifically, $V_h$ is defined on a finite partition $\calQ_h = \{Q_h^1,\ldots,Q_h^{N_h}\}$ of $\Omega$.
Then the optimization problem that actually must be solved reads
\begin{gather}\label{eq:inv_delta}
\begin{aligned}
\min_{u,w}\ & \|Tu - \ydel\|_{Y}  \\
\st\    & Au + f_0 uw = f_1, \\
        & w(x) \in \Wadh \coloneqq 
        \Wad\cap V_h.
\end{aligned}\tag{I$_\delta$}
\end{gather}
Another regularizing mechanism comes into play by means of bound constraints in the spirit of the so-called method of quasi solutions \cite{ClasonKlassen:2018,DombrovskajaIvanov65,Ivanov62,Ivanov63,IvanovVasinTanana02,KaltenbacherKlassen:2018,LorenzWorliczek13,NeubauerRamlau14,SeidmanVogel89}. While this is exploited in several
spots, our focus lies on the regularization by discretization
methodology.
\revision{Note that \eqref{eq:inv_delta} only contains a discretization of the coefficient $w$. For an actual implementation, also evaluation of the norm in the objective function and of the state $u$ needs to be incorporated; we refer to  \cref{sec:compexp} on the latter.}
\Margin{R2 2.}

Our algorithmic challenge for solving \eqref{eq:inv} is twofold: 
\begin{itemize}
	\item \emph{the discretization level $h$ must be adapted to the noise level $\delta$} to allow us to recover $u^\dagger$, $w^\dagger$
	as $\delta$, $h(\delta) \to 0$ \cite{EnglHankeNeubauer:1996,Kirsch2021} and
	\item \emph{the problem \eqref{eq:inv_delta} must be solved to global optimality} for given pairs $(\delta, h(\delta))$.
\end{itemize}
Regarding the latter, solving \eqref{eq:inv_delta} to global optimality is not possible by gradient-based nonlinear programming (NLP) algorithms
due to the nonlinear dependence $w \mapsto u$ through the state equation, which renders \eqref{eq:inv} and \eqref{eq:inv_delta} nonconvex.
We therefore require a global optimization algorithm like (spatial) branch-and-bound 
\cite{tawarmalani2013convexification,Land1960,HabeckPfetschUlbrich2019}, which requires means to compute (approximate) primal (upper) and
dual (lower) bounds on \eqref{eq:inv_delta}. While primal bounds can be computed with the aforementioned NLP algorithms,
dual bounds require a convex relaxation. To this end, we can replace the product $uw$ by a variable $z$ and the four corresponding
\emph{linear} McCormick inequalities \cite{mccormick1976computability}
\[ 
\mcc(u, w, z) \coloneqq 
\left(
\begin{aligned}
u_\ell w + u w_\ell - u_\ell w_\ell - z\\ 
u_u w + u w_u - u_u w_u - z\\
z - u_u w - u w_\ell + u_u w_\ell\\
z - u_\ell w - u w_u + u_\ell w_u\\
\end{aligned}
\right)
\le 0 \text{ a.e.\ in } \Omega.
\]
Here, it is a priori known that $u_\ell \le u \le u_u$ and $w_\ell \le w \le w_u$ hold for some  $u_\ell$, $u_u$, $w_\ell$, $w_u \in L^\infty(\Omega)$ and all feasible pairs  $(u,w)$ for \eqref{eq:inv_delta} or \eqref{eq:inv},
that is, if $u$ solves the state equation
\begin{gather}\label{eq:pde}
Au + f_0uw = f_1.
\end{gather}
\revision{To this end, we make use of the a priori bounds on the coefficient $w$ imposed by $\Wad$ and can assume these to result in a priori bounds on $u$, e.g., via PDE theory; tightening these bounds on the state improves the McCormick inequalities to better approximate the equality constraint $z=wu$. 
}
\Margin{R2 3.}

This yields the \emph{convex} minimization problem
\begin{gather}\label{eq:mcc_delta}
\begin{aligned}
\min_{u,w,z}\ & \|Tu - \ydel\|_{Y}  \\
\st\    & Au + f_0 z = f_1, \\
        & \mcc(u,w,z)(x) \le 0 \text{ for a.e.\ } x \in \Omega,\\
        & w \in \conv \Wadh, z \in \Zad,
\end{aligned}\tag{McC$_\delta$}
\end{gather}
where $\Zad$ is a bounded subset of $L^\infty(\Omega)$
such that $uw \in \Zad$ holds for all feasible pairs
$(u,w)$ of \eqref{eq:inv_delta} or \eqref{eq:inv}.
\revision{As mentioned above, inital estimates on $u_\ell$, $u_u$ can typically be derived analytically , cf. \eqref{eq:pde_bounds_uniform} in \cref{ass:approximation_relaxation}. However, these bounds are extremely conservative and also computationally useless, since they involve unknown constants from Sobolev imbeddings or the Poincar\'{e} inequality and its generalizations, cf., e.g.,  \cite{leyffer2025mccormick}.
Consequently, these bounds must be sharpened
}
\Margin{R2 4.}
through techniques like optimization-based bound tightening (OBBT) \cite{gleixner2017three}. This yields
a sequence of linear programs that need to be solved in several rounds. The number of
linear programs is proportional to the number of McCormick inequalities so that it depends
linearly on the number of nodes of our discretization grid if we choose to discretize $u$ with
a finite-element ansatz with a nodal basis for $u$. Since this is computationally intractable
for any meaningful discretization, we employ recent ideas from \cite{leyffer2025mccormick}
and replace the inequality system $\mcc(u, w, z) \le 0$ by its locally-averaged counterpart
\[ 
\mcc_\tau^i(u, w, z) \coloneqq 
\left(
\begin{aligned}
u_\ell^i (P_\tau w) + (P_\tau u) w_\ell^i - u_\ell^i w_\ell^i - z_i\\ 
u_u^i (P_\tau w) + (P_\tau u) w_u^i - u_u^i w_u^i - z_i\\
z_i - u_u^i (P_\tau w) - (P_\tau u) w_\ell^i + u_u^i w_\ell^i\\
z_i - u_\ell^i (P_\tau w) - (P_\tau u) w_u^i + u_\ell^i w_u^i\\
\end{aligned}
\right)
\le 0 \text{ a.e.\ in } Q^i_\tau \text{ for all } i \in \{1,\ldots,N_\tau\},
\]
where $\calQ_\tau \coloneqq \{Q_\tau^1,\ldots,Q_\tau^{N_{\tau}}\}$ is a partition of
$\Omega$ on which a finite-dimensional
ansatz space $\Zadtau$ that captures $(P_\tau w)(P_\tau u)$
is defined so that $z \in \Zadtau$. Because
$(P_\tau w)(P_\tau u)$ is constant per $Q_\tau^i$,
the piecewise constant functions on the partition
$\calQ_\tau$ are the canonical ansatz space for $\Zadtau$
but higher-order ansatz spaces are conceivable too.
In this way, one obtains McCormick inequalities for the 
(partially) locally-averaged state equation
\begin{gather}\label{eq:pde_tau}
 A u + f_0 (P_\tau u) (P_\tau w) = f_1.
\end{gather}
We obtain the convex problem with finitely many linear inequality constraints
\begin{gather}\label{eq:mcc_delta_tau}
\begin{aligned}
\min_{u,w,z}\ & \|Tu - \ydel\|_{Y} \\
\st\    & Au + f_0 z = f_1, \\
        & \mcc^i_\tau(u,w,z)(x) \le 0 \text{ for a.e.\ } x\in  Q^i_\tau \text{ for all } i \in \{1,\ldots,N_\tau\},\\
        & w \in \conv \Wadh, z \in \Zadtau,
\end{aligned}\tag{McC$_{\delta\tau}$}
\end{gather}
which is much better suited for optimization-based bound-tightening since the number of linear programs
that need to be solved can be controlled through $\tau$. The problem \eqref{eq:mcc_delta_tau}
itself can be interpreted as the McCormick relaxation of the following problem
\begin{gather}\label{eq:inv_delta_tau}
\begin{aligned}
\min_{u,w}\ & \|Tu - \ydel\|_{Y} \\
\st\    & Au + f_0 (P_\tau u)(P_\tau w) = f_1, \\
& w(x) \in \Wadh \coloneqq \{
w \in V_h : w \in \Wad \},
\end{aligned}\tag{I$_{\delta\tau}$}
\end{gather}

We analyze the approximation relationship
that is sketched in \cref{fig:tasks}.
\begin{figure}
\centering
\begin{tikzpicture}
\filldraw[black!10!white]  (2,0) rectangle (6,1);
\node[align=left] at (4,0.5) {\eqref{eq:inv}};
\draw[->,very thick] (4, 0) -- (4, -1);
\node[align=left,anchor=north west] at (.6125,0) {{\footnotesize noisy observations} \\ {\footnotesize regularization by disc.}};
\node[align=left,anchor=north west] 
	at (4,-0.2)
	{\footnotesize \cref{prp:reg_dis_cor}};
\filldraw[black!10!white] (2,-2) rectangle (6,-1);
\node[align=left] 
	at (4,-1.5)
	{\eqref{eq:inv_delta}};
\filldraw[black!10!white] (-3,-4) rectangle (1,-3);
\node[align=left] 
	at (-1,-3.5)
	{\eqref{eq:inv_delta_tau}};
\draw[->,very thick] (4, -2) -- (4, -3);
\node[align=left,anchor=north west] 
	at (2.125,-2)
	{{\footnotesize McCormick}\\ {\footnotesize relaxation}};
\node[align=left,anchor=north west] 
	at (-0.5,-2.25)
	{\footnotesize \cref{prp:approx_tau}};

\filldraw[black!10!white]  (2,-4) rectangle (6,-3);
\node[align=left] 
	at (4,-3.5)
	{\eqref{eq:mcc_delta}};
\node[align=left,anchor=north west] 
	at (2.125,-3.3)
	{\footnotesize convex};
\node[align=left,anchor=north west] 
	at (-1.7,-1.625)
	{{\footnotesize local ave-}\\ {\footnotesize raging}};
\node[align=left,anchor=north west] 
	at (4,-2.25)
	{\footnotesize \cref{prp:relax_delta}};
	
\filldraw[black!10!white]  (2,-6) rectangle (6,-5);
\node[align=left] 
	at (4,-5.5)
	{\eqref{eq:mcc_delta_tau}};
\node[align=left,anchor=north west] 
	at (2.125,-5.3)
	{\footnotesize convex};
\node[align=left,anchor=north west] 
	at (-2.25,-4.5)
	{{\footnotesize McCormick}\\ {\footnotesize relaxation}};
\node[align=left,anchor=north west] 
	at (-0.25,-4.5)
	{\footnotesize \cref{prp:relax_tau}};

\draw[very thick,->,bend right,looseness=0.75] (2,-1.5) to (-1,-3);
\draw[very thick,->,bend right,looseness=1] (-1,-4) to (2,-5.5);

\node[align=left,anchor=north west] 
at (-2.25,1.25) {{\footnotesize\textcolor{darkgreen}{\textbf{Approximate solutions}}}};

\draw[darkgreen,very thick,->,bend left,looseness=3.5,out=90,in=90] (2,-1.5) to (2,0.5);
\draw[darkgreen,very thick,->,bend left,looseness=1,out=90,in=150] (-3,-3.5) to (2,-1.5);
\node[align=left,anchor=north west] 
	at (-2,-0.125) {{\footnotesize\textcolor{darkgreen}{$h(\delta)$, $\delta \to 0$}}};
\node[align=left,anchor=north west] 
	at (-2.625,-1.25) {{\footnotesize\textcolor{darkgreen}{$\tau \to 0$}}};

\node[align=left,anchor=north west] 
at (6.25,1.25) {{\footnotesize\textcolor{ceruleanblue}{\textbf{Approximate lower bounds}}}};

\draw[ceruleanblue,very thick,<-,bend left,out=90,in=90,looseness=2.5] 
	(6,.5) to (6,-3.5);
\draw[ceruleanblue,very thick,<-,bend left,out=90,in=90,looseness=2.5]
	(6,-1.5) to (6,-5.5);
\draw[ceruleanblue,very thick,<-,bend left,out=90,in=90,looseness=2.5] 
	(6,.5) to (6,-5.5);
\node[align=left,anchor=north west] 
	at (9.0625,-2.25)
	{\footnotesize\textcolor{ceruleanblue}{$\delta,\tau \to 0$}};
\node[align=left,anchor=north west]
	at (7.4,-4.25)
	{\footnotesize\textcolor{ceruleanblue}{$\tau \to 0$}};
\node[align=left,anchor=north west] 
	at (7.4,-0.25)
	{\footnotesize\textcolor{ceruleanblue}{$\delta \to 0$}};
	
\node[align=left,anchor=north west] 
	at (6.1,-0.6125)
	{\footnotesize \cref{prp:I_approx_lb_McCdelta}};
	
\node[align=left,anchor=north west] 
	at (6.1,-3.9)
	{\footnotesize \cref{prp:Idelta_approx_lb_McCdeltatau}};
	
\node[align=left,anchor=north west] 
	at (10.33,-2.25)
	{\footnotesize \cref{prp:I_approx_lb_McCdeltatau}};

\end{tikzpicture}
\caption{Regularization and approximation relationship between the introduced problems.
}\label{fig:tasks}
\end{figure}

\medskip

The remainder of this paper is organized as follows. After a brief collection of some notation and useful facts in \cref{sec:notation}, in \cref{sec:approx-relax}, we prove that \eqref{eq:inv} is approximated by the noisy and regularized problem \eqref{eq:inv_delta} and
\eqref{eq:inv_delta} admits the convex McCormick relaxations \eqref{eq:mcc_delta}, whose dimension can be reduced by means of local averaging as
in \eqref{eq:mcc_delta_tau}. In \cref{sec:approximate_lower_bounds}, we show that the solutions to \eqref{eq:mcc_delta} and 
\eqref{eq:mcc_delta_tau} directly provide approximate lower bounds for the problems
\eqref{eq:inv} and \eqref{eq:inv_delta}. Using these estimates in an error balancing argument, we derive conditions on the discretization
parameters $h$ and $\tau$ as well as on the optimization precision, under which a convergence rate in terms of the noise level can be proven.
Numerical experiments in \cref{sec:compexp} illustrate the performance of the proposed method and the beneficial effect of OBBT
even without executing a very expensive spatial branch-and-bound algorithm by using the solution to \eqref{eq:mcc_delta_tau} to initialize
an NLP solver. 
Some auxiliary abstract results on regularization by discretization of nonlinear ill-posed operator equations are proven in the appendix.

\section{Notation}\label{sec:notation}
For $f \in L^1(\Omega)$, we denote its total variation seminorm by $|\cdot|_{\TV} : L^1(\Omega) \to [0,\infty]$. Similarly,
we denote the \revision{(Gagliardo)} seminorm of the Sobolev space $W^{k,p}(\Omega)$ by $|\cdot|_{W^{k,p}}$ while the norm is denoted by
$\|\cdot\|_{W^{k,p}}$, \revision{cf., e.g.,  \cite{NezzaPalatucciValdinoci:2012}.}
\Margin{R2 5.}

The notation $a \sim b(a)$ is used as a shortcut for existence of constants $c_1$, $c_2 > 0$ such that $c_1 a \le b(a) \le c_2 a$.

$\mathcal{M}_f$ denotes the multiplication operator induced by some function $f:\Omega\mapsto\mathbb{R}$. From Kato--Ponce--Vega estimates  
(see, e.g., \cite{KenigPonceVega1993} for the case of non-integer $t>0$)
\begin{equation}\label{eq:KatoPonceVega}
\begin{aligned}
&\|f\,g\|_{H^t(\Omega)}\leq C(t,\Omega)\Bigl(\|f\|_{H^t(\Omega)}\|g\|_{L^\infty(\Omega)}
+\|g\|_{H^t(\Omega)}\|f\|_{L^\infty(\Omega)}\Bigr)\,, \\ 
&f, \, g\in H^t(\Omega)\cap L^\infty(\Omega)\,.
\end{aligned}
\end{equation}
we can conclude boundedness of $\|\mathcal{M}_f\|_{H^{t}\cap L^\infty \to H^{t}} < \infty$ for $f\in H^{t}\cap L^\infty$ for any $t\geq0$; in case $t>d/2$ we can skip the intersection with $L^\infty(\Omega)$ due to the embedding $H^{t}(\Omega)\hookrightarrow L^\infty(\Omega)$.

In part of the analysis we will make use of the parameter-to-observation map
$F = T \circ S$, where $S:w\mapsto(A+\mathcal{M}_{f_0w})^{-1}f_1$ is the (nonlinear) parameter-to-state map for the constraining differential equation. While $F$ is not continuously invertible as a mapping $L^2(\Omega)\to Y$, its restriction to a finite dimensional subspace $V_h$ is, 
\revision{provided $F$ s injective. This is due to the fact that the $\sup$ in the definition of the stability constant}
\Margin{R2 6.}
\begin{equation}\label{eq:kappa_for_inv}
\kappa_h \ge \sup_{w\in \Wadh \setminus\bigl\{P_{\Wadh}w^\dagger\bigr\}}\frac{\bigl\|w-P_{\Wadh}w^\dagger\bigr\|_{L^2(\Omega)}}{\bigl\|F(w)-F(P_{\Wadh}w^\dagger)
	\bigr\|_Y}.
\end{equation}
\revision{is in fact attained as a maximum over the finite dimensional and bounded, thus compact set $\Wadh \setminus\bigl\{P_{\Wadh}w^\dagger\bigr\}$.
}
In \eqref{eq:kappa_for_inv}, $w^\dagger$ denotes the exact solution to the inverse problem with exact data, that is, $F(w^\dagger)=y^\dagger$.

\section{Approximation and Relaxation Relations}\label{sec:approx-relax}
In this section, we prove that \eqref{eq:inv} is approximated by the noisy and regularized
problem \eqref{eq:inv_delta} and \eqref{eq:inv_delta} admits the convex McCormick relaxations 
\eqref{eq:mcc_delta}. Moreover, the dimension of \eqref{eq:mcc_delta} can be reduced
by means of local averaging in the underlying PDE and the McCormick inequalities,
which yields \eqref{eq:mcc_delta_tau}. 

In the approximation results below, we work with two levels of generality. While \cref{prp:approx_tau}, \cref{prp:relax_delta}, \cref{prp:relax_tau} 
hold under the quite general set of assumptions \cref{ass:approximation_relaxation}, the quantifying results \cref{prp:reg_dis_cor}, \cref{prp:htau_h_upper_bound}, \cref{thm:balancing} need a more concrete setting.

\begin{assumption}\label{ass:approximation_relaxation}
For fixed noise level $\delta > 0$
and mesh sizes $\tau$, $h(\delta) > 0$,
it holds that:
\begin{enumerate}[label=(\arabic*)]
\item
$V_h$ is a closed, convex subset of a
finite-dimensional ansatz space for $L^2(\Omega)$.
\item\label{itm:pde_sol} For all $w \in \Wad$, \eqref{eq:pde}
and \eqref{eq:pde_tau} admit unique solutions in $U$.
\item\label{itm:pde_box_ass} For all $w \in \Wad$, the solution $u$
to \eqref{eq:pde} satisfies
\begin{gather}\label{eq:pde_bounds_uniform}
   u_\ell(x) \le u(x) \le u_u(x)
   \text{ for a.e.\ } x \in \Omega,
\end{gather}
where $u_\ell$, $u_u \in L^\infty(\Omega)$.
\item\label{itm:pde_tau_box_ass} For all $w \in \Wad$, the solution $u$ to
\eqref{eq:pde_tau} satisfies
\begin{gather}\label{eq:pde_tau_bounds_uniform}
u_\ell^i \le (P_\tau u)(x) \le u_u^i
\text{ for a.e.\ } x \in Q_\tau^i
\text{ for all } i \in \{1,\ldots,N_\tau\},
\end{gather}
where $u_\ell^i$, $u_u^i \in \R$ for all 
$i \in \{1,\ldots,N_\tau\}$.
%
%
where $w_\ell$, $w_u \in L^\infty(\Omega)$.
\item\label{itm:Wadh_box_ass} For all $w \in \Wadh$, it holds that
\[ w_\ell^i \le (P_\tau w)(x) \le w_u^i
\text{ for a.e.\ } x \in Q_\tau^i
\text{ for all } i \in \{1,\ldots,N_\tau\},
\]
where $w_\ell^i$, $w_u^i \in \R$ for all
$i \in \{1,\ldots,N_\tau\}$.
\item\label{itm:Wad_box_proj}
For all $w \in \Wad$, it holds that
$P_{V^h} w = P_{\Wadh} w$.
\item\label{itm:Zad_ass} $\Zad$ is closed, convex, and
bounded in $L^\infty(\Omega)$.
If $u$ solves \eqref{eq:pde} for $w \in \Wadh$, then $uw \in \Zad$.
\item\label{itm:Zadtau_ass} $\Zadtau$  is closed, convex, and
bounded in $L^\infty(\Omega)$ and consists of piecewise
constant functions on $\calQ_\tau$. If $u$ solves \eqref{eq:pde_tau} for
$w \in \Wadh$, then $(P_\tau u)(P_\tau w) \in \Zadtau$.
\end{enumerate}
When sequences of meshes with $h \to 0$ or $\tau \to 0$
are considered, we additionally assume that
\begin{enumerate}[resume,label=(\arabic*)]
\item there exist $c_1$, $c_2 > 0$, independently of $h$ and $\tau$
such that for all $Q_h^i \in \calQ_h$ and all
      $Q_\tau^i \in \calQ_\tau$, there exists a ball $B \subset \R^d$ such that
      $Q_h^i \subset B$ and $c_1 |B| \le |Q_h^i|$, respectively
      $Q_\tau^i \subset B$ and $c_2 |B| \le |Q_\tau^i|$, holds.
\end{enumerate}
\end{assumption}

\begin{remark}
The boundedness assumed in \cref{ass:approximation_relaxation} 
\ref{itm:Zad_ass} and \ref{itm:Zadtau_ass}
can be ensured by additional bound constraints on $z$ that
are compatible with the feasible
set structure. In particular, we assume uniform bounds
on $w$ in \cref{ass:approximation_relaxation}  \ref{itm:Wadh_box_ass}
and that \eqref{eq:pde} and \eqref{eq:pde_tau} imply
uniform bounds on $u$ in
\cref{ass:approximation_relaxation} \ref{itm:pde_box_ass}
and \ref{itm:pde_tau_box_ass}. Pairwise multiplication
pointwise a.e.\ directly gives such
bound constraints on $z$.

We note that an alternative is to add the uniform
bounds in \eqref{eq:pde_bounds_uniform} 
and \eqref{eq:pde_tau_bounds_uniform}
as state constraints to the problem formulations 
\eqref{eq:mcc_delta} and \eqref{eq:mcc_delta_tau}.
\end{remark}

\begin{remark}
We note that \cref{ass:approximation_relaxation} \ref{itm:Wad_box_proj} 
\revision{may not}
\Margin{R2 7.}
be satisfied
if $\Wad$ and thus $\Wadh$ are nonconvex. In this case, one may first relax $\Wad$ to $\conv \Wad$
and then apply the derivation of McCormick relaxations and approximation results using $\conv \Wad$
instead of $\Wad$.
\revision{In this case, one then may need to apply a potentially expensive global optimization algorithm such as branch-and-bound
to close the gap to the unrelaxed problem, which may generally be large.}
\Margin{R1 3.}
\end{remark}

\paragraph{\eqref{eq:inv_delta} approximates \eqref{eq:inv}}

Since our analysis relies on error estimates on Sobolev spaces, in this paragraph we concretize the setting to $U\subseteq H^{2s}(\Omega)$ for some $s\in(0,1]$. This, for $s\in(0,1)$  includes, e.g., the fractional Laplacian $A=(-\Delta)^s$ equipped with Dirichlet boundary conditions $u=g$ in $\mathbb{R}^d\setminus\Omega$ or for $s=1$ any elliptic differential operator $Au=\nabla\cdot (M\nabla u)$ with uniformly positive definite coefficient matrix $M\in C^1(\Omega;\mathbb{R}^{d\times d})$, again equipped with Dirichlet boundary conditions. Other boundary conditions could be incorporated in a straightforward manner, but at the cost of additional technicalities that we avoid here for clarity of exposition.
\revision{
On the one hand,}
\Margin{R2 8.} 
this allows to vary the degree of ill-posedness of the inverse problem. On the other hand, due to its nonlocal character in the $s\in(0,1)$
case, the
inverse problem with a single observation exhibits uniqueness even from observations on an arbitrary small open subset of
$\mathbb{R}^d$~\cite{GhoshRuelandSaloUhlmann:2020,Rueland:2021}, while in case $s=1$ the full Neumann-to-Dirichlet map would be required for this purpose.\footnote{While going further into detail about uniqueness questions would be beyond the scope of this paper, we point to the fact that that uniqueness from a single partial observation can also be achieved by incorporating appropriate a priori knowledge into the admissible set $\Wad$.}

The discretization space $V_h$ is defined by the piecewise polynomials on a shape-regular partition $\calQ_h$ of $\Omega$, with $h$ being the maximal 
diameter of the elements. We here focus on piecewise constant functions, as this matches the projection spaces used for local averaging.

To prove convergence of a solution to \eqref{eq:inv_delta} to a solution of \eqref{eq:inv}, 
we observe that \eqref{eq:inv_delta} corresponds to the so-called least squares version of regularization by projection and apply \cref{prp:regdis}.
To this end, we have the following two options.
\begin{itemize}
\item Reduced formulation $\xi=w$, $\mathbb{F}(\xi)=F(w)=TS(w)$. Here $S:w\mapsto u$ solving $Au + f_0 uw = f_1$ is the parameter-to-state map.
\item All-at-once formulation $\xi=(u,w)$, $\mathbb{F}(\xi)=(Au + f_0 uw, Tu)$. Here, the discretization space $\mathbb{X}_h$ incorporates both the projection of $w$
onto the piecewise constants and the FE discretization of $u$.
\end{itemize}
We first of all focus on the former and in \cref{rem:regdis_all-at-once} indicate the modifications required for the latter.

\begin{proposition}\label{prp:reg_dis_cor}
Let \cref{ass:approximation_relaxation} and $s\in(0,1]$ hold. 
\Margin{R2 9.}
Let  $A:U\subseteq H^{2s}(\Omega)\to L^2(\Omega)$
and $T:L^2(\Omega)\to Y$ be  isomorphisms, 
$A:L^2(\Omega)\to H^{-2s}(\Omega)$ bounded,
$f_0, 1/f_0\,\in H^{2s}(\Omega)\cap L^\infty(\Omega)$, $f_1\in L^2(\Omega)$, and let  $(w^\dagger,u^\dagger)\in \BV(\Omega)\times H^{2s}(\Omega)\cap L^\infty(\Omega)$ with $1/S(P_{\Wadh} w^\dagger)\in L^\infty(\Omega)$ for all $h>0$
be the exact solution to \eqref{eq:inv}, that is, $Tu^\dagger=y^\dagger$, $w^\dagger\in\Wad$, 
and denote by $(w_h^\delta,u_h^\delta)$ a solution to \eqref{eq:inv_delta}.
Then a choice 
\begin{equation}\label{eq:hdelta}
h(\delta)\sim \delta^{\max\big\{\frac{2}{1 + 4s},\frac{2d}{3d+4sd}\big\}}
\end{equation}
yields convergence with a rate 
\[
\|w_{h(\delta)}^\delta-w^\dagger\|_{L^2(\Omega)}=
\left\{
\begin{matrix}
O\left(\delta^{\frac{1}{1+4s}}\right)
& \text{ if } 3/2 + 2s/d - 2s > 1/2, \\
O\left(\delta^{\frac{3d+4s - 4sd}{3d + 4s}}\right)
& \text{ if } 3/2 + 2s/d - 2s \le 1/2 \\
\end{matrix}
\right.
\]
as $\delta \to 0$.
\end{proposition}
\begin{proof}
\Margin{R2 12.}
\revision{
We will use \cref{prp:regdis} and to this end need to verify the key estimate for convergence \eqref{eq:conv-cond}, which here reads
\begin{equation}\label{eq:conv-cond_wh}
\kappa_h \big\|F(P_{\Wadh} w^{\dagger}) - F(w^\dagger)\big\|_Y \to 0\text{ as } h \to 0
\end{equation}
Its verification} relies on one hand on an approximation error estimate. 
Here we have the estimates
\begin{equation}\label{eq:approx_w}
\|w - P_{V_h}w\|_{L^p}^p
\le \|w - P_{V_h}w\|_{L^{1}}\|w - P_{V_h}w\|_{L^\infty}^{p-1} 
\le \sqrt{d}\,h\,|w|_{\TV}\text{diam}(W)^{p - 1}
\end{equation}
for $w \in \Wad$,
where the first inequality is due to H\"older's inequality and the
second due to, e.g., Theorem 12.26, in particular (12.24), in
\cite{maggi2012sets}.
We actually have to bound $\|F(P_{V_h} w) - F(w)\|_{Y}$, cf. \eqref{eq:conv-cond}.
To this end, we make use of an approximation result 
\begin{equation}\label{eq:approx_w_2}
\|(I-P_{V_h})\psi\|_{L^{r^*}(\Omega)}\leq C \, h |\psi|_{W^{1,r^*}(\Omega)}
\leq C\, C_{H^{2s},W^{1,r^*}}^\Omega \, h \,
\|\psi\|_{H^{2s}(\Omega)}, 
\end{equation}
where $C_{H^{2s},W^{1,r^*}}^\Omega$ denotes the embedding constant of the continuous embedding $H^{2s}(\Omega)\hookrightarrow W^{1,r^*}(\Omega)$, cf.~\cite{AdamsFournier:2003} for $2s-d/2\geq -d/r^*$, that is, $1/r=1-1/r^*\leq 2s/d+1/2 \eqqcolon 1/r(s)$ in the estimate 
\begin{equation}\label{eqn:approx_err}
\|F(w_h) - F(w)\|_{Y}
\leq \|T\|_{L^2(\Omega)\to Y} \|(A+\mathcal{M}_{f_0w_h})^{-1}\mathcal{M}_{f_0S(w)}(w_h-w)\|_{L^2(\Omega)},
\end{equation}
Here we have used that 
$u_h-u \coloneqq S(w_h)-S(w)$ for $w_h = P_{V_h} w$ 
(noting that $P_{V_h} w= P_{\Wadh} w$ for $w\in\Wad$ by \cref{ass:approximation_relaxation} \ref{itm:Wad_box_proj})
solves
\[ 0=(Au_h + f_0 uw_h) -(Au + f_0 uw)= A(u_h-u) + f_0 w_h (u_h-u) + f_0u (w_h-w) \]
For $w_h=P_{V_h} w$, using an Aubin--Nitsche trick-type argument, 
\revision{cf., e.g., \cite[Sec. 2.3, Proposition 2.31]{ErnGuermond:2004},} 
\Margin{R2 10.}
we obtain
\begin{equation}\label{eq:AubinNitsche}
\begin{aligned}
\hspace{1em}&\hspace{-1em}\|(A+\mathcal{M}_{f_0w_h})^{-1}\mathcal{M}_{f_0S(w)}(w_h-w)\|_{L^2(\Omega)}\\
&=\sup_{\phi\in C^\infty(\Omega)\setminus\{0\}}
\frac{\int_\Omega (w_h-w) \mathcal{M}_{f_0S(w)} (A^*+\mathcal{M}_{f_0w_h})^{-1}\phi\,dx}{\|\phi\|_{L^2(\Omega)}}\\
&\revision{=\sup_{\phi\in C^\infty(\Omega)\setminus\{0\}}
\frac{\int_\Omega (w_h-w) (I-P_{V_h})\,\mathcal{M}_{f_0S(w)} (A^*+\mathcal{M}_{f_0w_h})^{-1}\phi\,dx}{\|\phi\|_{L^2(\Omega)}}
}\\
&\underset{\eqref{eq:approx_w_2}}\leq\|w_h-w\|_{L^r(\Omega)} \|\mathcal{M}_{f_0S(w)} (A^*+\mathcal{M}_{f_0w_h})^{-1}\|_{L^2(\Omega)\to H^{2s}(\Omega)} \, C\, C_{H^{2s},W^{1,r^*}}^\Omega \, h \\
&\underset{\eqref{eq:approx_w}}\lesssim h^{1+1/r(s)} = h^{\ell(s)}.
\end{aligned}
\end{equation}
with \revision{$1/r(s)=2s/d+1/2$, hence}
\Margin{R2 11.}
\begin{equation}\label{eqn:ell_s}
\ell(s)=3/2+2s/d.
\end{equation}

On the other hand, we need to estimate the stability constant $\kappa_h$ introduced in \eqref{eq:kappa_for_inv}. To this end, we first
of all derive Lipschitz stability of the inverse problem in an appropriate \emph{artificial} pre-image space topology and lift this
to the true image space topology by means of an inverse estimate. In the reduced setting we have, with $u=S(w)$
for $w$, $\tilde{w} \in \Wad$, $w \neq \tilde{w}$, 
\begin{equation}\label{eqn:FwtilFw_lower_0}
\begin{aligned}
&\|F(\tilde{w})-F(w)\|_{Y}= \|T(A+\mathcal{M}_{f_0\tilde{w}})^{-1}\,\mathcal{M}_{f_0u}[\tilde{w}-w]\|_{Y}\\
&\geq \|T^{-1}\|_{Y\to L^2(\Omega)}^{-1} \|A+\mathcal{M}_{f_0\tilde{w}}\|_{L^2(\Omega)\to H^{-2s}(\Omega)}^{-1}
\|\mathcal{M}_{1/(f_0u)}\|_{H^{-2s}(\Omega)\to H^{-2s}(\Omega)}^{-1}\, \|\tilde{w}-w\|_{H^{-2s}(\Omega)}.
\end{aligned}
\end{equation}
The required boundedness of the multiplication operator $\mathcal{M}_{1/(f_0u)}:H^{-2s}(\Omega)\to H^{-2s}(\Omega)$, $v\mapsto (1/(f_0u))\cdot v$
follows from $1/(f_0u) \in H^{2s}(\Omega)\cap L^\infty(\Omega)$, 
for $w=P_{\Wadh} w^\dagger$, due to our assumption $1/S(P_{\Wadh} w^\dagger)\in L^\infty(\Omega)$ for all $h>0$, cf., \eqref{eq:kappa_for_inv}. 
For boundedness of the multiplication operator $\mathcal{M}_{f_0\tilde{w}}:\,v\mapsto {f_0\tilde{w}}\cdot v$, we can use ${\tilde{w}}\in\Wad$, which due to the bounded constraints in $\Wad$ includes an $L^\infty$-bound on ${\tilde{w}}$, so that 
$\mathcal{M}_{\tilde{w}}$ is even bounded as an operator from $L^2(\Omega)$ into itself.

An inverse estimate reads
\[ \|\tilde{w} - w\|_{H^{-2s}(\Omega)} \ge h^{2s}\|\tilde{w} - w\|_{L^2(\Omega)};  \]
see Theorem 3.6 in \cite{graham2005finite} and yields
\begin{equation}\label{eqn:FwtilFw_lower}
\|F(\tilde{w})-F(w)\|_{Y}\gtrsim h^{k(s)} \|\tilde{w}-w\|_{L^2(\Omega)},
\end{equation}
with the choice
\begin{equation}\label{eqn:k_s}
k(s)=2s.
\end{equation}
Hence, \eqref{eq:kappa_for_inv} holds with $\kappa_h \sim h^{-k(s)}$.
\revision{Consequently, we can certify \eqref{eq:conv-cond_wh}
and thus use \eqref{eq:regdis_error_estimate} in \cref{prp:regdis}.
We} balance the $O(h^{\frac{1}{2}})$-term
$\|(P_{\Wadh} w^{\dagger}) - w^\dagger\|_{L^2(\Omega)}$,
the $O(h^{\ell(s) - k(s)})$-term
$2\kappa_h \|F(P_{\Wadh} w^{\dagger}) - F(w^\dagger)\|_Y$,
and the $O(h^{-k(s)}\delta)$-term $2\kappa_h\delta$,
where we can combine the first two terms into one
$O(h^{\min\{\frac{1}{2},\ell(s) - k(s)\}})$-term.
This yields the balanced choice
\[ h(\delta)^{\min\big\{\frac{1}{2} + k(s),\ell(s)\big\}}
   \sim \delta
   \quad
   \Leftrightarrow
   \quad 
   h(\delta) \sim 
   \delta^{\max\big\{\frac{2}{1 + 4s},\frac{2d}{3d+4sd}\big\}}.
\]
\end{proof}
We highlight that $\ell(1) = 3/2 + 2/d > 2 = k(1)$ in the argument above so that the important
(local) case $s = 1$ is feasible for all $d \in \{1,2,3\}$.
\begin{remark}
The condition $1/S(P_{\Wadh} w^\dagger)\in L^\infty(\Omega)$ for all $h>0$ follows from \cref{ass:approximation_relaxation} \ref{eq:pde_bounds_uniform} with $u_\ell(x)u_u(x)\geq c>0$ for a.e. $x \in\Omega$. The latter can be enforced, e.g., by a maximum principle with boundary data $g$ bounded away from zero under appropriate sign conditions on $w$ (imposed, e.g. via the constraints in $\Wad$) and $f_0$, cf. e.g., \cite{EvansBook}.
\end{remark}
\begin{remark}\label{rem:regdis_all-at-once}
In the all-at-once setting, the counterpart of \eqref{eqn:approx_err} also involves an FE error estimate on $u$, which can be obtained by assuming sufficient smoothness on the exact solution $u^\dagger$\footnote{we don't even need to derive this from elliptic regularity, since $u^\dagger$ is part of the solution to the inverse problem, on which we can make assumptions} and using sufficiently high order polynomials and/or a finer grid for $u$ as compared to the one for $w$.

\revision{For deriving lower bounds on the difference of function values of $F$}
\Margin{R2 13.}
that we need for bounding $\kappa_h$, we have, with $\mathbb{Y}=Y^{mod}\times Y$, 
\[
\|\mathbb{F}(\tilde{u},\tilde{w})-\mathbb{F}(u,w)\|_{Y^{mod}\times Y}= 
\|(A+\mathcal{M}_{f_0\tilde{w}})[\tilde{u}-u]+\mathcal{M}_{f_0u} [\tilde{w}-w]\|_{Y^{mod}}+\|T(\tilde{u}-u)\|_{Y}
\]
hence with $\|\cdot\|_{Y^{mod}}=c\|\cdot\|_{H^{-2s}(\Omega)}$ for $c$ small enough
\[
\begin{aligned}
&\|\mathbb{F}(\tilde{u},\tilde{w})-\mathbb{F}(u,w)\|_{Y^{mod}\times Y}\\
&\geq 
\left(\|T^{-1}\|_{Y\to L^2(\Omega)}^{-1} - c\,\|(A+\mathcal{M}_{f_0\tilde{w}})\|_{L^2(\Omega)\to H^{-2s}(\Omega)}\right)\, \|\tilde{u}-u\|_{L^2(\Omega)}\\
&\quad+c\,\|\mathcal{M}_{1/(f_0u)}\|_{H^{-2s}(\Omega)\to H^{-2s}(\Omega)}^{-1}\, \|\tilde{w}-w\|_{H^{-2s}(\Omega)}
\end{aligned}
\]
thus, with Theorem 3.6 in \cite{graham2005finite} as above and, similarly,
\begin{equation}\label{eq:inverse-estimate-u}
\|u\|_{L^2(\Omega)}\gtrsim h^{k(s)} \|u\|_{H^{2s}(\Omega)}, \quad u\in U_h,
\end{equation}
(which entails some requirements on the FE space for $u$; a choice $U_h \coloneqq \{v\in H^{2s}_0(\Omega)\,:\, v\vert_K\in\mathcal{P},\ K\in\calQ_h\}$ suffices for this purpose), we obtain 
\[
\|\mathbb{F}(\tilde{u},\tilde{w})-\mathbb{F}(u,w)\|_{Y^{mod}\times Y}\gtrsim h^{-k(s)} \left(\|\tilde{u}-u\|_{H^{2s}(\Omega)}+\|\tilde{w}-w\|_{L^2(\Omega)}\right),
\]
hence again the convergence condition \eqref{eq:conv-cond} holds if $\ell(s)>k(s)$ for $\mathbb{X}=H^{2s}(\Omega)\times L^2(\Omega)$, $\mathbb{Y}=H^{-2s}(\Omega)\times Y$.
\end{remark}

\begin{remark}
Since \cref{prp:reg_dis_cor} is about solutions to \eqref{eq:inv_delta}, we note that solutions exist by virtue of
the direct method of calculus of variations using the nonnegativity of $f_0$ to maintain the ellipticity.
\end{remark}

\paragraph{\eqref{eq:inv_delta_tau} approximates \eqref{eq:inv_delta}}
\begin{proposition}\label{prp:approx_tau}
Let \cref{ass:approximation_relaxation} hold with some reflexive space $U$ satisfying $U\stackrel{c}{\hookrightarrow}L^p(\Omega)$ and $U\hookrightarrow L^{\tilde{p}}(\Omega)$ for some $1\leq p<\tilde{p}\leq\infty$ and $A:U\to U^*$ be an isomorphism, $f_0\in L^{p/(p-1)}(\Omega)$, $f_1\in U^*$.

Then the problem \eqref{eq:inv_delta_tau} admits a
minimizer. The mapping $U\times\Wadh \ni (u, w) \mapsto (u_\tau, w) \in U\times\Wadh$
between the feasible sets of \eqref{eq:inv_delta} and \eqref{eq:inv_delta_tau},
where $u$ solves \eqref{eq:pde} and $u_\tau$ solves \eqref{eq:pde_tau}, is onto.

\revision{Then $u_\tau\rightharpoonup u$ in $U$ for $\tau\to0$.}
\Margin{R2 14.}
\end{proposition}
\begin{proof}
The existence of solutions for \eqref{eq:inv_delta_tau} follows with the same argument as for
\eqref{eq:inv_delta}, from the direct method of calculus of variations. 
Indeed, the feasible set is bounded in $U$ due to \cref{ass:approximation_relaxation},
\ref{itm:pde_box_ass}, \ref{itm:Wadh_box_ass}, and the estimate 
$\int_\Omega f_0 w u \phi\leq \|f_0\|_{L^{p/(p-1)}(\Omega)}\|w\|_{L^\infty(\Omega)}\|u\|_{L^\infty(\Omega)}\|\phi\|_{L^p(\Omega)}$ for $\phi\in U\hookrightarrow L^p(\Omega)$, and $w$, $u$ satisfying  \cref{ass:approximation_relaxation} \ref{itm:pde_box_ass}, \ref{itm:Wadh_box_ass}, along with the continuous invertibility of $A:U^*\to U$.
The feasible set is obviously also non-empty. 
Thus the existence of a minimizer follows with the weak lower semi-continuity of the objective, implied by boundedness of $T:U\mapsto Y$. 

Due to the respective uniqueness of solutions to \eqref{eq:pde} and \eqref{eq:pde_tau}
for fixed $w$ and the fact that $P_\tau$ is onto, the mapping is $(u,w) \mapsto (u_\tau, w)$ is onto.

Let $(u,w)$ be feasible for \eqref{eq:inv_delta_tau}. 
Then $P_\tau w \to w$ in $L^q(\Omega)$ for all $q \in[1,\infty)$ by virtue of Lebesgue's differentiation theorem and 
the bound constraints in $\Wad$.
Moreover, the uniform boundedness of $P_\tau w$ and $P_\tau u_\tau$ in $L^\infty(\Omega)$ due to \cref{ass:approximation_relaxation} \ref{itm:pde_box_ass}, \ref{itm:Wadh_box_ass}, hence of  $f_0 P_\tau w P_\tau u_\tau$ in $L^{p/(p-1)}(\Omega)\hookrightarrow U^*$ along with boundedness of $A^{-1}:U^*\to U$ imply uniform boundedness of $u_\tau = A^{-1}(f_1-f_0 P_\tau w P_\tau u_\tau)$ in $U$ and thus $u_\tau\weakto \bar{u}$ in $U$ for some $\bar{u}\in U$ with $P_\tau(u_\tau-\bar{u})\to0$ in $L^p(\Omega)$ by compactness of $U\stackrel{c}{\hookrightarrow}L^p(\Omega)$ and boundedness $P_\tau:L^p(\Omega)\to L^p(\Omega)$, after restricting to a subsequence. 
Due to \cref{ass:approximation_relaxation} \ref{itm:pde_box_ass}, we also have $\bar{u}\in L^\infty(\Omega)$, hence $P_\tau \bar{u}-\bar{u}\to 0$ in $L^q(\Omega)$ for all $q \in[1,\infty)$.
Likewise, uniform boundedness of $P_\tau(u_\tau-\bar{u})$ in $L^\infty(\Omega)$ and its convergence to zero in $L^p(\Omega)$, by 
\revision{interpolation of Lebesgue spaces}
\Margin{R2 15.}
also yields $P_\tau(u_\tau-\bar{u})\to0$ in $L^q(\Omega)$ for all $q \in[1,\infty)$.
As a consequence, we have $f_0w (P_\tau u_\tau-\bar{u})=f_0w (P_\tau(u_\tau-\bar{u})+P_\tau \bar{u}-\bar{u}) \to 0$ in $L^{\tilde{p}}(\Omega)^*\hookrightarrow U^*$.
By passing to the limit in \eqref{eq:pde_tau}, 
we thus obtain
\[ \langle A \bar{u} + f_0 w \bar{u}, v\rangle_{U^*,U} = \langle f_1, v\rangle_{U^*,U} \]
for all $v \in U$, which means that $\bar{u}$ solves \eqref{eq:pde} for 
$w$ so that $u = \bar{u}$ follows by uniqueness. Since this argument can be made
for all subsequences of $(u_\tau)_\tau$, the claim follows.
\end{proof}
To obtain quantitative information on the result of \cref{prp:approx_tau}, we observe that the only difference
is in the bilinear terms in \eqref{eq:pde} and \eqref{eq:pde_tau}. We consider the reduced formulations
\[ F(w) = TS(w),\enskip S : w \mapsto u \text{ solves } \eqref{eq:pde} \]
for \eqref{eq:inv_delta} and 
\[ F_\tau(w) = TS_\tau(w),\enskip S_\tau : w \mapsto u \text{ solves } \eqref{eq:pde_tau}. \]
We can derive an approximation relationship that will turn out useful later under the assumptions that the mesh size $\tau$
is chosen smaller than the mesh size $h$ and that the $\tau$-grid is embedded in the $h$-grid. This is formalized in the
following assumption.
\begin{assumption}\label{ass:grid_embedding}
	For all $Q_h \in \calQ_h$, there exist $Q_\tau^{i_1},\ldots,Q_\tau^{i_m} \in \calQ_\tau$, $m \in \N$ such that
	\[ Q_h = Q_\tau^{i_1} \cup \cdots \cup Q_\tau^{i_m}. \]
\end{assumption}
Our approximation result then reads.
\begin{proposition}\label{prp:htau_h_upper_bound}
	Let \cref{ass:approximation_relaxation,ass:grid_embedding} hold, 
    $s\in(0,1]$, $A:H_0^{2s}(\Omega)\to L^2(\Omega)$ an isomorphism, $T:L^2(\Omega)\to Y$ bounded, $f_0, 1/f_0\,\in H^{2s}(\Omega)\cap L^\infty(\Omega)$, $f_1\in L^2(\Omega)$,
	and $w_h \in \Wadh$. Then
	\begin{multline*}
	\|F_\tau(w_h) - F(w_h)\|_{Y}
	\le \\
\revision{C}
\|T\|_{L^2(\Omega)\to Y}
	\|\mathcal{M}_{f_0}\|_{H^{2s} \to H^{2s}} \|(A+\mathcal{M}_{f_0 w_h})^{-*}\|_{L^2(\Omega)\to H^{2s}(\Omega)}
	\|w_h\|_{L^\infty} 
|u_{h\,\tau}|_{H^{2s}} 
\revision{\,\tau^{4s}}
	\end{multline*}
	holds for some 
\revision{$C>0$ depending only on $\Omega$ and $s$.} 
\end{proposition}
\Margin{R2 16.}
\begin{proof}
	Let $u_{h\,\tau} = S_\tau(w_h)$, $u_h = S(w_h)$. We immediately have
	\[
	\|F_\tau(w_h) - F(w_h)\|_{Y}
	\le \|T\|_{L^2(\Omega)\to Y}\|u_{h\,\tau} - u_h\|_{L^2(\Omega)},
	\]
	To estimate $\|u_{h\,\tau} - u_h\|_{L^2(\Omega)}$ from above, we observe
	\[
	\begin{aligned}
	0 &= Au_{h\,\tau} + f_0 (P_\tau w_h)(P_\tau u_{h\,\tau})  -(Au_h + f_0 w_hu_h)\\
	&= A(u_{h\,\tau}-u_h) + f_0w_h (u_{h\,\tau} - u_h) - f_0 w_h u_{h\,\tau} + f_0 (P_\tau w_h)(P_\tau u_{h\,\tau})  \\
	&= (A +\mathcal{M}_{f_0w_h})(u_{h\,\tau}-u_h) + f_0 w_h(P_\tau u_{h\,\tau} - u_{h\,\tau}),
	\end{aligned}
	\]
	where we have used that $P_\tau w_h = w_h$ due to \cref{ass:grid_embedding}, which implies
	\[ u_{h\,\tau} -u_h
	= (A +\mathcal{M}_{f_0w_h})^{-1}\bigl(\mathcal{M}_{f_0w_h}(P_\tau u_{h\,\tau} - u_{h\,\tau})\bigr),
	\]
	where $A+\mathcal{M}_{f_0 w_h}$ is invertible by means of \cref{ass:approximation_relaxation} \ref{itm:pde_sol}.
	Using $P_\tau w_h = w_h$ again, we estimate
	\begin{align}
	\hspace{1em}&\hspace{-1em}\bigl\|(A +\mathcal{M}_{f_0w_h})^{-1}\mathcal{M}_{f_0w_h}(P_\tau  u_{h\,\tau} - u_{h\,\tau})\bigr\|_{L^2(\Omega)} \nonumber\\
	&=\sup_{\phi\in C^\infty(\Omega)\setminus\{0\}}
	\frac{\int_\Omega (P_\tau  u_{h\,\tau} - u_{h\,\tau}) \mathcal{M}_{f_0 w_h} (A+\mathcal{M}_{f_0 w_h})^{-*}\phi\,dx}{\|\phi\|_{L^2(\Omega)}}
	\nonumber\\
	&=\sup_{\psi\in C^\infty(\Omega)\setminus\{0\}}
	\frac{\int_\Omega (P_\tau w_h)(P_\tau  u_{h\,\tau} - u_{h\,\tau}) \psi\,dx}{\|(A+\mathcal{M}_{f_0w_h})^*\mathcal{M}_{f_0^{-1}} \psi\|_{L^2(\Omega)}}
	\nonumber\\
	&=\sup_{\psi\in C^\infty(\Omega)\setminus\{0\}}
	\frac{\int_\Omega (P_\tau w_h)(P_\tau  u_{h\,\tau} - u_{h\,\tau}) (\psi - P_\tau \psi)\,dx}{\|(A+\mathcal{M}_{f_0w_h})^*\mathcal{M}_{f_0^{-1}} \psi\|_{L^2(\Omega)}}
	\label{eq:nitsche}\\
	&\le \|\mathcal{M}_{f_0}\|_{H^{2s} \to H^{2s}} \|(A+\mathcal{M}_{f_0 w_h})^{-*}\|_{L^2(\Omega)\to H^{2s}(\Omega)}
	\|w_h\|_{L^\infty} C(\Omega,2s)^2 (\tau^{2s})^2 |u_{h\,\tau}|_{H^{2s}} 
	\nonumber,
	\end{align}
	where we have used that $(P_\tau w_h)(P_\tau u_{h\,\tau} - u_\tau)(P_\tau \psi) = 0$ for \eqref{eq:nitsche} and 
	\[ 
	\|g - P_\tau g\|_{L^2(\Omega)}\leq C(\Omega,\mu) \, \tau^\mu |g|_{H^\mu(\Omega)}
	\]
	for $g=u_{h\,\tau}$ and $g=\psi$ for some $C(\Omega,\mu) > 0$; see, e.g., \cite[Theorem 4.4.20]{BrennerScott:2008}.
\end{proof}

\paragraph{\eqref{eq:mcc_delta} relaxes \eqref{eq:inv_delta}}
\begin{proposition}\label{prp:relax_delta}
Let \cref{ass:approximation_relaxation} hold with some reflexive space $U$ satisfying $U{\hookrightarrow}L^p(\Omega)$  for some $1\leq p\leq\infty$ and $A:U\to U^*$ be an isomorphism, $f_0\in L^{p/(p-1)}(\Omega)$, $f_1\in U^*$.

Then \eqref{eq:mcc_delta} is a convex minimization problem that admits a minimizer. 

Let $(u,w) \in U\times \Wadh$ be feasible for \eqref{eq:inv_delta}.
Then $\big(u,w,uw\big) \in U\times \Wadh \times \Zad$ is feasible for \eqref{eq:mcc_delta}
and the objective values coincide.
\end{proposition}
\begin{proof}
\eqref{eq:mcc_delta} is convex because the objective of \eqref{eq:pde_tau} is 
convex, the sets $\conv \Wadh$ and $\Zad$ are convex, and all further constraints
are linear. The feasible set is bounded in $U$, due to \cref{ass:approximation_relaxation},
\ref{itm:pde_box_ass}, \ref{itm:Wadh_box_ass}, our assumption on $f_0$, $f_1$, and the continuous invertibility
of $A:U^*\to U$. It is non-empty since the feasible set of
\eqref{eq:inv_delta} is non-empty. 
Thus the existence of minimizers follows with the weak lower semi-continuity of the objective.
The final claim follows from \cref{ass:approximation_relaxation}
\ref{itm:pde_box_ass} and 
the bound constraints in $\Wad$,
which imply
\[
\mcc(u,w,uw)(x) \le 0
\text{ for a.e.\ } x \in\Omega.
\]
\end{proof}

\paragraph{\eqref{eq:mcc_delta_tau} relaxes \eqref{eq:inv_delta_tau}}
\begin{proposition}\label{prp:relax_tau}
Let \cref{ass:approximation_relaxation} hold with some reflexive space $U$ satisfying $U{\hookrightarrow}L^p(\Omega)$  for some $1\leq p\leq\infty$ and $A:U\to U^*$ be an isomorphism, $f_0\in L^{p/(p-1)}(\Omega)$, $f_1\in U^*$.

Then \eqref{eq:mcc_delta_tau} is a convex minimization problem that admits a minimizer. 

Let $(u,w) \in U\times \Wadh$ be feasible for \eqref{eq:inv_delta_tau}.
Then $\big(u,w,(P_\tau u)(P_\tau w)\big) \in U\times \Wadh \times \Zadtau$ is feasible for \eqref{eq:mcc_delta_tau}
and the objective values coincide.
\end{proposition}
\begin{proof}
The claim follows with an argument that parallels the proof of \cref{prp:relax_delta}.
\end{proof}

\section{Approximate Lower Bounds on \eqref{eq:inv} and \eqref{eq:inv_delta}} 
\label{sec:approximate_lower_bounds}
In this section, we briefly prove that the solutions to \eqref{eq:mcc_delta} and 
\eqref{eq:mcc_delta_tau} directly provide approximate lower bounds for the problems
\eqref{eq:inv} and \eqref{eq:inv_delta}, that is, true lower bounds after
subtracting an error due to the increased feasible set in \eqref{eq:mcc_delta}
and \eqref{eq:mcc_delta_tau} and the local averaging in \eqref{eq:mcc_delta_tau}.
We note that by optimization-based bound-tightening procedures 
\cite{leyffer2025mccormick}, we can shrink the $L^\infty$-bounds
in \cref{ass:approximation_relaxation} \ref{itm:pde_box_ass}
and \ref{itm:pde_tau_box_ass}. While this increases the optimal objective
\eqref{eq:mcc_delta} and \eqref{eq:mcc_delta_tau} and thus tightens the bound,
it does not decrease the subtracted approximation error since the latter
requires that an approximation of the optimal solution is in the feasible set of
\eqref{eq:mcc_delta} and \eqref{eq:mcc_delta_tau}, which might be cut off
by the bound-tightening procedure.

After the estimates, we proceed with an error balancing argument.

\paragraph{\eqref{eq:mcc_delta} yields approximate lower bound on \eqref{eq:inv}}

\begin{proposition}\label{prp:I_approx_lb_McCdelta}
Let $\delta$, $h(\delta) > 0$.
Let \cref{ass:approximation_relaxation} hold.
Let	$m_{\eqref{eq:inv}}$ and $m_{\eqref{eq:mcc_delta}}$ be the optimal
objective values for \eqref{eq:inv} and \eqref{eq:mcc_delta}
respectively. Then
\[ 
m_{\eqref{eq:inv}}
\ge m_{\eqref{eq:mcc_delta}} 
 - 
 \big(\inf\{ \|Tu - Tu_\delta\|_{Y} : (u,w) \in \argmin \eqref{eq:inv},
  u_\delta \text{ solves } \eqref{eq:pde} \text{ for } P_{h(\delta)} w\}
  + \delta\big),
\]	
Assuming further that the solution operator 
$w\mapsto(A+\mathcal{M}_{f_0 w})^{-1}f_1$ of 
\eqref{eq:pde} is Lipschitz as a mapping
$L^p(\Omega) \to U$, $p \in [1,\infty)$,
with Lipschitz constant $L_{\eqref{eq:pde}}$ and
$T$ has operator norm
$\|T\|_{U\to Y}$, we obtain
\[
m_{\eqref{eq:inv}}
\ge m_{\eqref{eq:mcc_delta}} 
- (L_{\eqref{eq:pde}}\|T\|_{U\to Y}
\inf\{\|w - P_{h(\delta)}w\|_{L^p(\Omega)} : (u,w) \in \argmin\eqref{eq:inv}\}
+ \delta)
\]
Assuming further that $w \in \BV(\Omega)$ for some
$(u,w) \in \argmin \eqref{eq:inv}$, we obtain
\[ m_{\eqref{eq:inv}}
\ge m_{\eqref{eq:mcc_delta}} 
- \big(L_{\eqref{eq:pde}}\|T\|_{U\to Y}
\|w_u - w_\ell\|_{L^\infty(\Omega)}^{\frac{p-1}{p}}\sqrt{d}^{\frac{1}{p}}
\inf\big\{|w|_{\TV}^{\frac{1}{p}} : (u,w) \in \argmin\eqref{eq:inv}\big\}h(\delta)^{\frac{1}{p}}
+ \delta\big).
\]
\end{proposition}
\begin{proof}
Let $(u,w)$ solve \eqref{eq:inv} and $u_\delta$ solve \eqref{eq:pde} for
$w_\delta \coloneqq P_{h(\delta)} w$. Then $(u_\delta, w_\delta, u_\delta w_\delta)$ is feasible for \eqref{eq:mcc_delta} by virtue of
\cref{ass:approximation_relaxation} \ref{itm:pde_box_ass} and
\ref{itm:Zad_ass}, implying
\[ \|Tu_\delta - \ydel\|_{Y} \ge m_{\eqref{eq:mcc_delta}}. \]
We have
\[
\|T u_\delta - \ydel\|_{Y} - \|T u - y^\dagger\|_{Y}
\le \|T u - T u_\delta\|_{Y} + \delta,
\]
which yields the first claim. The second claim then follows from
\[ \|Tu - Tu_\delta\|_{Y} \le L_{\eqref{eq:pde}} \|T\|_{U\to Y}\|w - P_{h(\delta)}w\|_{L^p(\Omega)}
\]
and the third from \eqref{eq:approx_w}.
\end{proof}

\paragraph{\eqref{eq:mcc_delta_tau} yields approximate lower bound on \eqref{eq:inv_delta}}

\begin{proposition}\label{prp:Idelta_approx_lb_McCdeltatau}
Let $\delta$, $h(\delta) > 0$, $\tau > 0$. Let \cref{ass:approximation_relaxation} hold.
Let	$m_{\eqref{eq:inv_delta}}$ and $m_{\eqref{eq:mcc_delta_tau}}$ be the optimal
objective values for \eqref{eq:inv_delta} and \eqref{eq:mcc_delta_tau}
respectively. Then
\[ 
m_{\eqref{eq:inv_delta}}
\ge m_{\eqref{eq:mcc_delta_tau}} 
- \inf\{ \|Tu - Tu_{\tau}\|_{Y} : (u,w) \in \argmin \eqref{eq:inv_delta},
u_\tau \text{ solves } \eqref{eq:pde_tau} \text{ for } w\}.
\]	
\end{proposition}
\begin{proof}
Let $(u,w)$ solve \eqref{eq:inv_delta} and $u_\tau$ solve \eqref{eq:pde_tau} for
$w$. Then $(u_\tau, w, (P_\tau u_\tau) (P_\tau w))$ is feasible for
\eqref{eq:mcc_delta_tau} by virtue of \cref{ass:approximation_relaxation}
\ref{itm:pde_box_ass} and \ref{itm:Zadtau_ass}, implying
\[ \|Tu_\tau - \ydel\|_{Y} \ge m_{\eqref{eq:mcc_delta_tau}}. \]
We have
\[
\|Tu_\tau - \ydel\|_{Y} - \|Tu - \ydel\|_{Y}
\le \|T u_\tau - T u\|_{Y},
\]
which yields the claim.
\end{proof}

\paragraph{\eqref{eq:mcc_delta_tau} yields approximate lower bound on \eqref{eq:inv}}

\begin{proposition}\label{prp:I_approx_lb_McCdeltatau}
Let $\delta$, $h(\delta) > 0$, $\tau > 0$.
Let \cref{ass:approximation_relaxation} hold
Let	$m_{\eqref{eq:inv}}$ and $m_{\eqref{eq:mcc_delta_tau}}$ be the optimal
objective values for \eqref{eq:inv} and \eqref{eq:mcc_delta_tau}
respectively. Then
\[ 
m_{\eqref{eq:inv}}
\ge m_{\eqref{eq:mcc_delta_tau}} 
- 
\big(\inf\{ \|Tu - Tu_{\delta\tau}\|_{Y} : (u,w) \in \argmin \eqref{eq:inv},
u_{\delta\tau} \text{ solves } \eqref{eq:pde_tau} \text{ for } P_{h(\delta)} w\}
+ \delta\big).
\]	
\end{proposition}
\begin{proof}
	Let $(u,w)$ solve \eqref{eq:inv} and $u_{\delta\tau}$ solve \eqref{eq:pde} for
	$w_\delta \coloneqq P_{h(\delta)} w$. Then $(u_{\delta\tau}, w_\delta, (P_\tau u_{\delta\tau})(P_\tau w_\delta))$ is feasible for \eqref{eq:mcc_delta_tau} by virtue of
	\cref{ass:approximation_relaxation} \ref{itm:pde_tau_box_ass},
	\ref{itm:Zad_ass}, and \ref{itm:Zadtau_ass}, implying
	\[ \|Tu_{\delta\tau} - \ydel\|_{Y} \ge m_{\eqref{eq:mcc_delta_tau}} \]
	so that the claim follows as for \cref{prp:I_approx_lb_McCdelta}.
\end{proof}

\paragraph{Error-balanced approximate solutions to \eqref{eq:inv}}
We are ultimately interested in a balanced error estimate of the form
\[ \|w^\dagger - \hat{w}_{\delta\tau}\|_{L^2(\Omega)} = O(\delta^t), \]
where $(u^\dagger, w^\dagger)$ is optimal for \eqref{eq:inv} and $(\hat{u}_{\delta\tau},\hat{w}_{\delta\tau})$ is
feasible for \eqref{eq:inv_delta_tau} and $\varepsilon$-optimal in terms of the objective, that is,
\begin{equation}\label{eq:epsilon} 
\varepsilon \coloneqq \|T\hat{u}_{\delta\tau} - \ydel\|_{Y} - \ell_{\eqref{eq:inv_delta_tau}}, 
\end{equation}
where $\ell_{\eqref{eq:inv_delta_tau}}$ is a computed lower bound of $\eqref{eq:inv_delta_tau}$,
$\ell_{\eqref{eq:inv_delta_tau}} \le m_{\eqref{eq:inv_delta_tau}}$. Assuming, we can solve the convex problem
\eqref{eq:mcc_delta_tau} to global optimality, we can have $m_{\eqref{eq:mcc_delta_tau}} \le \ell_{\eqref{eq:inv_delta_tau}}$
Thus $\varepsilon$ is a quantity that we can prescribe to a global optimization
algorithm like spatial branch-and-bound that tightens $\ell_{\eqref{eq:inv_delta_tau}}$ by successively restricting
\eqref{eq:mcc_delta_tau}. In addition, we assume that the noise level $\delta$ is given or can at least
be estimated. Thus, the goal is to determine the mesh sizes $h$, $\tau$, and $\varepsilon$ so that we can
estimate and balance the different error terms contributing to $\|w^\dagger - \hat{w}_{\delta\tau}\|_{L^2(\Omega)}$
as $\delta^t$ for some computable $t > 0$ that ideally is as large as possible.

\begin{theorem}\label{thm:balancing}
Let \cref{ass:approximation_relaxation,ass:grid_embedding} and $s\in(0,1]$ hold. 
Let $A:U\subseteq H^{2s}(\Omega)\to L^2(\Omega)$ 
and $T:L^2(\Omega)\to Y$ be  isomorphisms, 
$A:L^2(\Omega)\to H^{-2s}(\Omega)$ bounded,
$f_0, 1/f_0\,\in H^{2s}(\Omega)\cap L^\infty(\Omega)$, $f_1\in L^2(\Omega)$, and let  $(w^\dagger,u^\dagger)\in \BV(\Omega)\times H^{2s}(\Omega)\cap L^\infty(\Omega)$ with $1/S(P_{\Wadh} w^\dagger)\in L^\infty(\Omega)$ for all $h>0$
be the exact solution to \eqref{eq:inv}, that is, $Tu^\dagger=y^\dagger$, $w^\dagger\in\Wad$.
and denote by $(\hat{w}_{\delta\tau},\hat{u}_{\delta\tau})$ an approximate  solution to \eqref{eq:inv_delta} in the sense of \eqref{eq:epsilon} with $\hat{u}_{\delta\tau}=S(\hat{w}_{\delta\tau})$.

If we choose 
\begin{equation}\label{eq:choice_balancing}
\varepsilon=O(\delta),\quad
\tau=O\bigl(\min\{\delta^{1/(4s^2)},h(\delta)\}\bigr),\quad
h \sim \delta^{\max\{\frac{2}{1 + 4s},\frac{2d}{3d+4sd}\}}
\end{equation}
and have $\sup_{h\,\tau} \|(A+\mathcal{M}_{f_0 w_h})^{-*}\|_{L^2(\Omega)\to H^{2s}(\Omega)} < \infty$,
$\sup_{h\,\tau} \|S_\tau(w_h)\|_{H^{2s}(\Omega)} < \infty$, we obtain the balanced convergence rate
\[ \|w^\dagger - \hat{w}_{\delta\tau}\|_{L^2(\Omega)}
= \left\{
\begin{matrix}
O\left(\delta^{\frac{1}{1+4s}}\right)
& \text{ if } 3/2 + 2s/d - 2s > 1/2, \\
O\left(\delta^{\frac{3d+4s - 4sd}{3d + 4s}}\right)
& \text{ if } 3/2 + 2s/d - 2s \le 1/2 \\
\end{matrix}
\right.
\]
as $\delta \to 0$.
\end{theorem}
\begin{proof}
\cref{prp:regdis_tau} yields the estimate
\[
\|w^\dagger - \hat{w}_{\delta\tau}\|_{L^2(\Omega)} 
\leq C\Bigl(h^{1/2}+ h^{-k(s)}\bigl(
h^{\ell(s)}+\tau^{k(s)^2}+\delta+\varepsilon\bigr)\Bigr)
\] 
with $\ell(s)=3/2 + 2s/d$, $k(s)=2s$ and some $C > 0$, where the $\tau^{k(s)^2}$-estimate is due to
\cref{prp:htau_h_upper_bound}. As in \cref{prp:reg_dis_cor}, we seek for balanced choices of $h$
(and $\tau$ and $\varepsilon$) for a given noise level $\delta > 0$.
Given such a noise level $\delta>0$, the optimal balancing depends on whether $O(h^{1/2})$-term
or the $O(h^{\ell(s) - k(s)})$-term dominates and on the restriction $\tau \le h$ since the $\tau$-grid
is embedded into the $h$-grid.

Clearly, we can easily balance
$\varepsilon$ and $\tau^{k(s)^2}$ with $\delta$ by choosing
\[ \varepsilon(\delta) = O(\delta)
   \quad\text{and}\quad
   \tau = O(\min\{\delta^{1/(4s^2)},h(\delta)\}),
\]
where we have explicitly taken care of the requirement
$\tau \le h(\delta)$ due to \cref{ass:grid_embedding}.
We determine $h(\delta)$ next.

Again, we combine the $O(h^{1/2})$-term and the $O(h^{\ell(s) - k(s)})$-term into one
$O(h^{\min\{\frac{1}{2},\ell(s) - k(s)\}})$-term.
It thus remains to balance this with the $O(h^{-k(s)}\delta)$-term.
This is the situation from the proof of \cref{prp:reg_dis_cor}, yielding the choice
\[ h(\delta) \sim \delta^{\max\big\{\frac{2}{1 + 4s},\frac{2d}{3d+4sd}\big\}}. \]
Then, the terms $O(h(\delta)^{\min\{\frac{1}{2},\ell(s) - k(s)\}})$, 
$O(h(\delta)^{-k(s)}\varepsilon)$,
and $O(h(\delta)^{-k(s)}\delta)$ are in balance and dominating, which yields the claimed
convergence rate (again similar to \cref{prp:reg_dis_cor}).
\end{proof}

\section{Computational Experiment}\label{sec:compexp}
For our computational experiment, we consider the problem
\eqref{eq:inv}, where $Au + f_0 u w = f_1$ is realized
by the weak formulation
\[ \int_0^1 \partial_x u \partial_x v\,\mathrm{d}x
   + \int_0^1 f_0 u w v\,\mathrm{d}x
   = \int_0^1 f_1 v\,\mathrm{d}x\quad\text{ for } v \in H^1_0(0,1)
\]
with $f_0 = 36$ and $f_1(x) = 50 \sin(2 \pi x)^2$ for 
$x \in (0,1)$. We have $T = \operatorname{id}$
and set $w^\dagger(x) \coloneqq \cos(2\pi x)^2$. Apart from the
specific parameter values, this is the benchmark problem from
\cite{leyffer2025mccormick}. We thus choose $w_\ell = 0$ and
$w_u = 1$. We can choose $s = 1$ according to the results in 
\cref{thm:balancing}.

\paragraph{Experimental setup}
We solve the PDE on a uniform discretization
into $1024$ grid cells with a finite-element ansatz of
first-order continuous Lagrange elements
and prescribe the sequence of noise values
$\delta \in \{10^{-1}, 10^{-2}, 10^{-3}, 10^{-4}, 10^{-5} \}$.
We generate $\ydel$ by perturbing the nodal values $y^\dagger$
randomly using a normal distribution with mean zero and
standard deviation $1.1\,\delta$, implying the actually
observed measurement errors tabulated in
\cref{tbl:noise_implications}.
We compute the mesh sizes $h$ and $\tau$ using the balancing
result \cref{thm:balancing} as
\begin{gather}\label{eq:balancing_implemented}
h = \lceil\delta^{-\max\{2/(1+4s),2/(3+4s)\}}\rceil^{-1}
\enskip\text{and}\enskip
\tau = \lceil\min\{h,\delta^{1/(4s^2)}\}^{-1}\rceil^{-1},
\end{gather}
where the rounding up in between a double inversion is employed
to get an integer number of grid cells.
The resulting values and the corresponding numbers of grid cells
$N_h$ and $N_\tau$ are tabulated in \cref{tbl:noise_implications}
too.
\begin{table}[h]
\begin{center}
\caption{Derived values of $h$ and $\tau$,
corresponding grid cell numbers $N_h$ and $N_\tau$,
and observed measurement errors for the prescribed noise levels $\delta$.}\label{tbl:noise_implications}
\begin{tabular}{cccccc}
	\toprule
	$\delta$ 
	& $h$ 
	& $N_h$
	& $\tau$
	& $N_\tau$
	& $\|y^\dagger - \ydel\|_{L^2}$ \\
	\midrule
	$10^{-1}$ 
	& $3.981 \times 10^{-1}$
	& $3$
	& $3.981 \times 10^{-1}$
	& $3$
	& $8.566 \times 10^{-2}$ \\
	$10^{-2}$ 
	& $1.585 \times 10^{-1}$
	& $7$
	& $1.585 \times 10^{-1}$
	& $7$
	& $9.137 \times 10^{-3}$ \\
	$10^{-3}$ 
	& $6.310 \times 10^{-2}$
	& $16$
	& $6.310 \times 10^{-2}$
	& $16$	
	& $8.699 \times 10^{-4}$ \\
	$10^{-4}$ 
	& $2.512 \times 10^{-2}$ 
	& $40$
	& $2.512 \times 10^{-2}$ 
	& $40$
	& $8.880 \times 10^{-5}$ \\
	$10^{-5}$ 
	& $1.000 \times 10^{-2}$ 
	& $101$
	& $1.000 \times 10^{-2}$ 
	& $101$	
	& $9.245 \times 10^{-6}$ \\
	\bottomrule
\end{tabular}
\end{center}
\end{table}
We run the L-BFGS-B algorithm \cite{liu1989limited} on 
\eqref{eq:inv_delta}, yielding an approximately stationary point.
We initialize L-BFGS-B using a computed solution to the McCormick
relaxation \eqref{eq:mcc_delta_tau}, where the used values for the
bounds on the locally averaged $u_\ell^i$ and $u_u^i$
for all $i$ in \eqref{eq:mcc_delta_tau} are generated by applying
the optimization-based bound
tightening (OBBT) procedure described in \cite{leyffer2025mccormick}
to the initial conservative bounds $u_\ell^i = -10^{3}$ and
$u_u^i = 10^3$ for all $i$. For all prescribed noise levels $\delta$,
this procedure implies a suboptimality $\varepsilon$ by
comparison to the global lower bound of zero that is already below
the threshold choice $\delta$ for the suboptimality;
see \cref{thm:balancing}. This allows us to be in the regime
of the balancing result \cref{thm:balancing} without the need
to execute an expensive spatial branch-and-bound procedure for
global optimization. The observed values are tabulated in 
\cref{tbl:opt_results} as $\varepsilon_{\textrm{O}}$.

For comparison and in particular to highlight the 
benefit of \emph{tight} lower bound computations, we 
also run L-BFGS-B with two further initializations. 
First, we initialize with a solution to the McCormick 
relaxation \eqref{eq:inv_delta_tau}
without OBBT but the conservative bounds $u_\ell^i = -10^{3}$ and
$u_u^i = 10^3$ for all $i$. Second, we initialize with
the constant function $w(x) = 0.5$ for all $x \in (0,1)$.
The observed values are tabulated as 
$\varepsilon_{\textrm{P}}$ and $\varepsilon_{\textrm{N}}$ in 
\cref{tbl:opt_results} too.

\paragraph{Implementation and compute environment}
The experiment code was written in Python.
For the implementation of L-BFGS-B, we have used the one
in SciPy \cite{2020SciPy-NMeth}. The other optimizations
were performed by using the primal simplex method implemented
in Gurobi \cite{gurobi}. All computations were executed on a laptop computer
with a 2.50\,GHz Intel(R) Core(TM) i7-11850H CPU
and 64 GB RAM.

\paragraph{Results}
Since this is used several times in the description of the results, we note that
we generally refer to the three different initializations of L-BFGS-B as
\begin{itemize}
	\item[$\mathrm{O}$] Solution of the McCormick relaxation with OBBT.
	\item[$\mathrm{P}$] Solution of the McCormick relaxation without OBBT.
	\item[$\mathrm{N}$] Constant function with value $0.5$.
\end{itemize}

Regarding the observed values for $\varepsilon$, the initialization
$\mathrm{O}$ almost always performs
best but the desired  suboptimality threshold of $\delta$ due to \cref{thm:balancing} is
satisfied for all three initializations and all values of $\delta$.

Next, we compare the main quantity of interest, the resulting
reconstruction errors $\|w^\dagger - \hat{w}^{\textrm{IV}}_{\delta\tau}\|_{L^2(0,1)}$
for $\textrm{IV} \in \{\textrm{O}, \textrm{P}, \textrm{N}\}$.
Here, $\hat{w}^{\textrm{IV}}_{\delta\tau}$ is the 
the best respective computed solution candidate for \eqref{eq:inv}
computed with L-BFGS-B with initialization indicated by
$\textrm{IV}$. Here, the initialization $\mathrm{O}$ performs best or gives a
value close to the best value. In addition and in contrast to the other two initializations,
a clear downward trend can be observed here as $\delta$ tends to zero. Specifically, the reconstruction error approximately
halves when the noise level is divided by $10$. Again, the obtained values are tabulated
in \cref{tbl:opt_results}. For the smallest noise level, the reconstruction error obtained
with $\mathrm{O}$ is almost a magnitude smaller than the reconstruction errors obtained
with $\mathrm{P}$ and $\mathrm{N}$.
The observed rate of the reconstruction error for the case $\mathrm{O}$ is 
slightly higher (better) than the predicted rate from \cref{thm:balancing}.
To provide a visual impression of this effect,
the reconstruction error trends of $\|w^\dagger - \hat{w}^{\textrm{IV}}_{\delta\tau}\|_{L^2(0,1)}$
are plotted in comparison to $\delta^{1/(1+4s)} = \delta^{1/5}$,
the predicted rate of \cref{thm:balancing}, in
\cref{fig:error_trend_vs_delta}. A qualitative impression can be obtained
from the resulting solution candidates of L-BFGS-B, \revision{which are plotted
in \cref{fig:reconstructions_1e-3} for $\delta = 10^{-3}$ and \cref{fig:reconstructions_1e-5}
for $\delta = 10^{-5}$.} The function $\hat{w}^{\textrm{O}}_{\delta\tau}$
closely follows the globally optimal solution $w^\dagger$ while
$\hat{w}^{\textrm{P}}_{\delta\tau}$ is substantially off at several intervals
and $\hat{w}^{\textrm{N}}_{\delta\tau}$ is clearly an (approximately)
stationary point that has a substantially different shape than $w^\dagger$.\Margin{R1 4.}

\revision{Regarding the robustness of the selection of $h$ and $\tau$, we fix
them according to \eqref{eq:balancing_implemented} for a reference noise
level $\delta_\ast = 10^{-3}$ and then execute L-BFGS-B for the initializations 
$\textrm{IV} \in \{\textrm{O}, \textrm{P}, \textrm{N}\}$
for the sequence of noise values 
$\delta \in \{10^{-1},10^{-2},10^{-3},10^{-4},10^{-5}\}$. We observe that both
the suboptimality and the reconstruction error have a generally decreasing trend 
until the noise level $\delta = 10^{-3}$, for which $\tau$ and $h$ are computed using
our balancing result. Then, the values stagnate or increase slightly
with the exception of $\varepsilon_{\textrm{N}}$, which falls to the smallest
overall value for the suboptimality for $\delta = 10^{-4}$ and then increases 
back to a comparatively large value for $\delta = 10^{-5}$. Since the reconstruction
error is not close to the global optimum for this value, this outlier 
is a local minimum with a very low objective value / very good suboptimality
as can happen in non-convex minimization.
We have tabulated the recorded values in \cref{tbl:stability_results}.}\Margin{R1 2.}

We also briefly compare the run times of the different steps of the
optimization for the initial $\mathrm{O}$, $\mathrm{P}$, and
$\mathrm{N}$. The largest run time effect by far has the OBBT procedure
to tighten \eqref{eq:mcc_delta_tau} which only occurs in the case $\mathrm{O}$
and increases the run time of $\mathrm{O}$ over the one $\mathrm{P}$ by an order
of magnitude. The solution of the (convex) McCormick relaxation \eqref{eq:mcc_delta_tau}
in contrast is almost negligible and always about an order of magnitude smaller
than the L-BFGS-B execution so that the run time of $\mathrm{P}$ is comparable
with the one of $\mathrm{N}$. A breakdown of the run times can be looked up in
\cref{tbl:runtimes}.

\begin{table}[h]
	\begin{center}
		\caption{Observed suboptimality $\varepsilon_{\textrm{IV}}$
		and reconstruction error
		$\|w^\dagger - \hat{w}^{\textrm{IV}}_{\delta\tau}\|_{L^2}$
		for $\textrm{IV} \in \{\textrm{O},\textrm{P},\textrm{N} \}$
		for the prescribed noise levels $\delta$.
		The lowest values (up to our rounding precision of four
		digits) are bold faced.}\label{tbl:opt_results}
		\begin{adjustbox}{width=\textwidth}	
		\begin{tabular}{ccccccc}
			\toprule
			$\delta$ 
			& $\varepsilon_{\textrm{O}}$ 
			& $\varepsilon_{\textrm{P}}$ 
			& $\varepsilon_{\textrm{N}}$
			& $\|w^\dagger - \hat{w}^{\textrm{O}}_{\delta\tau}\|_{L^2}$
			& $\|w^\dagger - \hat{w}^{\textrm{P}}_{\delta\tau}\|_{L^2}$
			& $\|w^\dagger - \hat{w}^{\textrm{N}}_{\delta\tau}\|_{L^2}$
			\\
			\midrule
			$10^{-1}$ 
			& $\mathbf{4.125 \times 10^{-3}}$
			& $\mathbf{4.125 \times 10^{-3}}$
			& $\mathbf{4.125 \times 10^{-3}}$
			& $\mathbf{3.794 \times 10^{-1}}$
			& $\mathbf{3.794 \times 10^{-1}}$
			& $\mathbf{3.794 \times 10^{-1}}$
			\\
			$10^{-2}$ 
			& $\mathbf{6.947 \times 10^{-5}}$
			& $6.955 \times 10^{-5}$
			& $6.952 \times 10^{-5}$
			& $1.845 \times 10^{-1}$
			& $1.845 \times 10^{-1}$
			& $\mathbf{1.844 \times 10^{-1}}$
			\\
			$10^{-3}$ 
			& $\mathbf{4.896 \times 10^{-7}}$
			& $1.002 \times 10^{-6}$
			& $8.996 \times 10^{-7}$
			& $\mathbf{8.536 \times 10^{-2}}$
			& $2.345 \times 10^{-1}$
			& $1.532 \times 10^{-1}$
			\\
			$10^{-4}$ 
			& $1.359 \times 10^{-7}$ 
			& $\mathbf{1.042 \times 10^{-7}}$
			& $4.731 \times 10^{-7}$
			& $3.777 \times 10^{-2}$
			& $\mathbf{3.532 \times 10^{-2}}$
			& $1.517 \times 10^{-1}$
			\\
			$10^{-5}$ 
			& $\mathbf{4.980 \times 10^{-8}}$ 
			& $4.998 \times 10^{-7}$
			& $4.799 \times 10^{-7}$
			& $\mathbf{2.128 \times 10^{-2}}$
			& $1.552 \times 10^{-1}$
			& $1.483 \times 10^{-1}$
			\\
			\bottomrule
		\end{tabular}
		\end{adjustbox}
	\end{center}
\end{table}

\begin{table}[h]
	\begin{center}
		\caption{Observed suboptimality $\varepsilon_{\textrm{IV}}$
			and reconstruction error
			$\|w^\dagger - \hat{w}^{\textrm{IV}}_{\delta\tau}\|_{L^2}$
			for $\textrm{IV} \in \{\textrm{O},\textrm{P},\textrm{N} \}$
			for the prescribed noise levels $\delta$
			but fixed values for $h(\delta_\ast)$, $\tau(\delta_\ast)$ computed using
			\eqref{eq:balancing_implemented} for $\delta_\ast = 10^{-3}$.}\label{tbl:stability_results}
		\begin{adjustbox}{width=\textwidth}	
			\begin{tabular}{ccccccc}
				\toprule
				$\delta$ 
				& $\varepsilon_{\textrm{O}}$ 
				& $\varepsilon_{\textrm{P}}$ 
				& $\varepsilon_{\textrm{N}}$
				& $\|w^\dagger - \hat{w}^{\textrm{O}}_{\delta\tau}\|_{L^2}$
				& $\|w^\dagger - \hat{w}^{\textrm{P}}_{\delta\tau}\|_{L^2}$
				& $\|w^\dagger - \hat{w}^{\textrm{N}}_{\delta\tau}\|_{L^2}$
				\\
				\midrule
				$10^{-1}$ 
				& $3.625 \times 10^{-3}$
				& $3.628 \times 10^{-3}$
				& $3.635 \times 10^{-3}$
				& $1.723 \times 10^{-1}$
				& $2.704 \times 10^{-1}$
				& $2.026 \times 10^{-1}$
				\\
				$10^{-2}$ 
				& $4.097 \times 10^{-5}$
				& $4.134 \times 10^{-5}$
				& $4.128 \times 10^{-5}$
				& $8.710 \times 10^{-2}$
				& $2.387 \times 10^{-1}$
				& $1.522 \times 10^{-1}$
				\\
				$10^{-3}$ 
				& $4.896 \times 10^{-7}$
				& $1.002 \times 10^{-6}$
				& $8.996 \times 10^{-7}$
				& $8.536 \times 10^{-2}$
				& $2.345 \times 10^{-1}$
				& $1.532 \times 10^{-1}$
				\\
				$10^{-4}$ 
				& $6.138 \times 10^{-7}$
				& $1.044 \times 10^{-6}$
				& $2.515 \times 10^{-7}$
				& $8.808 \times 10^{-2}$
				& $2.411 \times 10^{-1}$
				& $1.510 \times 10^{-1}$
				\\
				$10^{-5}$ 
				& $6.205 \times 10^{-7}$
				& $2.203 \times 10^{-6}$
				& $1.758 \times 10^{-6}$
				& $8.812 \times 10^{-2}$
				& $2.505 \times 10^{-1}$
				& $1.695 \times 10^{-1}$
				\\
				\bottomrule
			\end{tabular}
		\end{adjustbox}
	\end{center}
\end{table}

\begin{figure}[ht]
\begin{center}
\pgfplotsset{width=.75\textwidth,height=5.25cm} 
\begin{tikzpicture}
\begin{axis}[
xmode=log,
ymode=log,
ytick={0.01,0.1,1},
xlabel=$\delta$,
ylabel=$\|w^\dagger - \hat{w}_{\delta\tau}^{\mathrm{IV}}\|_{L^2(\Omega)}$,
legend pos=south east,
]
\addplot[black,thick,mark=*] table {
	1e-1 3.794e-1
	1e-2 1.845e-1
	1e-3 8.536e-2
	1e-4 3.777e-2
	1e-5 2.128e-2
};
\addplot[black,thick,dashed,mark=diamond] table {
	1e-1 3.794e-1
	1e-2 1.845e-1
	1e-3 2.345e-1
	1e-4 3.532e-2
	1e-5 1.552e-1
};
\addplot[black,thick,dash dot,mark=star] table {
	1e-1 3.794e-1
	1e-2 1.844e-1
	1e-3 1.532e-1
	1e-4 1.517e-1
	1e-5 1.483e-1
};
\addplot[black,thick,densely dotted,mark=+] table {
	1e-1 0.6309573444801932
	1e-2 0.3981071705534972
	1e-3 0.251188643150958
	1e-4 0.15848932
	1e-5 0.1
};
\legend{$\mathrm{IV} = O$,
	$\mathrm{IV} = P$,
	$\mathrm{IV} = N$,
	$\delta^{1/(1 + 4s)}$}
\end{axis}
\end{tikzpicture}	
\end{center}
\caption{Reconstruction errors
$\|w^\dagger - \hat{w}_{\delta\tau}^{\mathrm{IV}}\|_{L^2(\Omega)}$ as $\delta \to 0$ for $\mathrm{IV} \in \{\mathrm{O},\mathrm{P},\mathrm{N}\}$.
}\label{fig:error_trend_vs_delta}
\end{figure}

\begin{figure}[ht]
\begin{center}
\includegraphics[width=\textwidth]{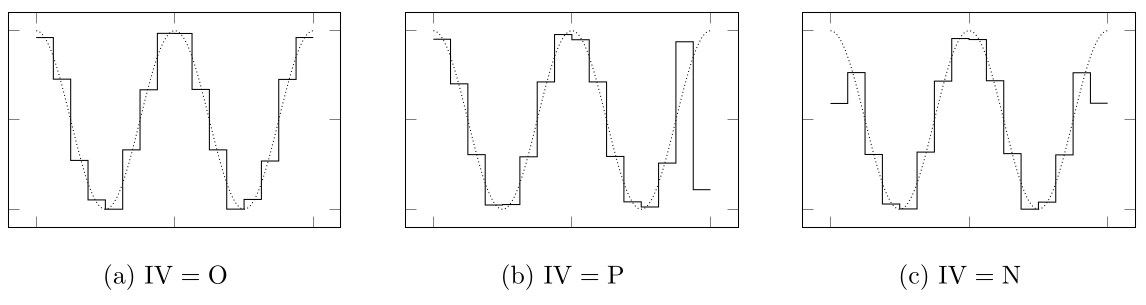}
\caption{$w^\dagger$ (dotted) vs.\ $\hat{w}_{\delta\tau}^{\mathrm{IV}}$ (solid)
	for $\delta =10^{-3}$ over $\Omega = (0,1)$.}\label{fig:reconstructions_1e-3}
\end{center}
\end{figure}

\begin{figure}[ht]
\begin{center}
\includegraphics[width=\textwidth]{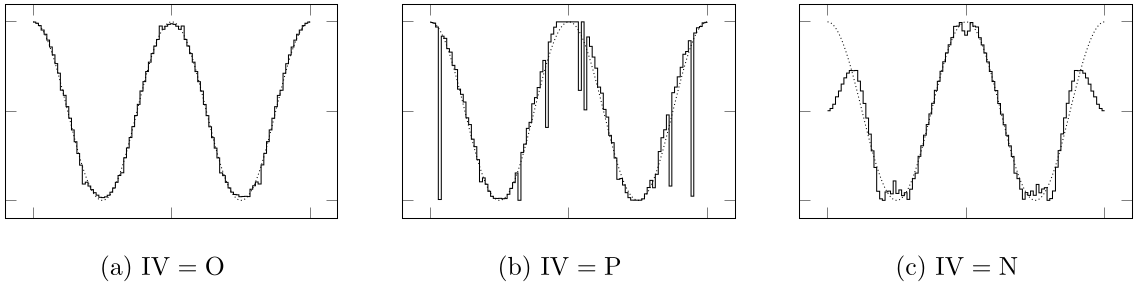}
\caption{$w^\dagger$ (dotted) vs.\ $\hat{w}_{\delta\tau}^{\mathrm{IV}}$ (solid)
	for $\delta =10^{-5}$ over $\Omega = (0,1)$.}\label{fig:reconstructions_1e-5}
\end{center}
\end{figure}

\begin{table}[h]
	\begin{center}
		\caption{Run times of OBBT (only $\mathrm{O}$), solution of the McCormick
		relaxation \eqref{eq:mcc_delta_tau} (only $\mathrm{O}$ and $\mathrm{P}$), and L-BFGS-B,
		and accumulated for $\mathrm{IV} \in \{\mathrm{O},\mathrm{P},\mathrm{N}\}$
		in seconds.}\label{tbl:runtimes}
		\begin{adjustbox}{width=\textwidth}	
			\begin{tabular}{ccccccccc}
				\toprule
				\multicolumn{1}{c|}{}
				& \multicolumn{4}{c|}{$\mathrm{O}$}
				& \multicolumn{3}{c|}{$\mathrm{P}$}
				& $\mathrm{N}$
				\\
				\multicolumn{1}{c|}{$\delta$}
				& OBBT
				& \eqref{eq:mcc_delta_tau} 
				& L-BFGS-B
				& \multicolumn{1}{c|}{$\sum$}
				& \eqref{eq:mcc_delta_tau} 
				& L-BFGS-B
				& \multicolumn{1}{c|}{$\sum$}
				& L-BFGS-B / $\sum$\\
				\midrule
				$10^{-1}$ 
				& $4.367 \times 10^{-1}$
				& $6.983 \times 10^{-2}$
				& $2.078$
				& $2.586$
				& $1.164 \times 10^{-1}$
				& $1.997$
				& $2.113$
				& $1.632$
				\\
				$10^{-2}$ 
				& $1.044$
				& $7.458 \times 10^{-2}$
				& $4.810$
				& $5.928$
				& $6.881 \times 10^{-2}$
				& $4.318$
				& $4.387$
				& $4.097$
				\\
				$10^{-3}$ 
				& $2.831$
				& $7.244 \times 10^{-2}$
				& $5.059$
				& $7.962$
				& $8.312 \times 10^{-2}$
				& $7.807$
				& $7.891$
				& $6.382$
				\\
				$10^{-4}$ 
				& $1.523 \times 10^{1}$
				& $1.130 \times 10^{-1}$
				& $3.377$
				& $1.872 \times 10^{1}$
				& $1.079 \times 10^{-1}$
				& $3.141$
				& $3.249$
				& $7.040$
				\\
				$10^{-5}$ 
				& $5.415 \times 10^1$
				& $1.571 \times 10^{-1}$
				& $5.359$
				& $5.967 \times 10^1$
				& $1.782 \times 10^{-1}$
				& $5.219$
				& $5.398$				
				& $7.270$
				\\
				\bottomrule
			\end{tabular}
		\end{adjustbox}
	\end{center}
\end{table}
\Margin{R2 17.}

\paragraph{Interpretation} Including tight lower bound information (case $\mathrm{O}$)
into the initialization of a local optimization procedure for \eqref{eq:mcc_delta_tau}
has improved the resulting reconstruction error substantially and we observe a behavior
close to the prediction of \cref{thm:balancing} when driving $\delta \to 0$ even though the assumptions of boundedness away from zero at
the boundary are not satisfied here.
In contrast to this, we do not observe a decreasing trend as $\delta \to 0$
when very weak (case $\mathrm{P}$) or no lower bound information (case $\mathrm{N}$)
is included into the solution process. We stress again that the observed $\varepsilon$ is
smaller than $\delta$ so that this assumption of \cref{thm:balancing} also holds for
$\mathrm{P}$ and $\mathrm{N}$ but we clearly do not know the constants for the big-O-terms
and thus believe that our results are plausible and we do not suspect any bugs.

\revision{Regarding our experiment for different levels of $\delta$ with
fixed and thus unbalanced values of $h$ and $\tau$, we observe a decreasing trend
of the suboptimality and reconstruction error until the $\delta$ is in balance
with $h$ and $\tau$, after which a sideways trend follows. This is a plausible
observation and supports the intuition that one a smaller choice for $h$ and $\tau$ 
is a sensible hedge if the noise level is not known very well.}\Margin{R1 2.}

The computation of the tight lower bound information is the performance bottleneck in
our experiment. This observation is in line with the common knowledge in global optimization
that the computation of good primal bounds (that is upper bounds in minimization) is
often possible in a relatively inexpensive manner while the computation of tight dual bounds (that
is lower bounds in minimization), which is also generally key to certify (approximate)
global optimality, may generally be expensive when the underlying problem is nonconvex.

In our experiment, we were able to access the dual information indirectly through the
solution of \eqref{eq:mcc_delta_tau} and were thus also able to bypass a full spatial
branch-and-bound algorithm. Such a full spatial branch-and-bound procedure is clearly
desirable but comes with a wealth of further degrees of freedom and options 
that can and need to be adjusted carefully but are complementary to the subject of
this article so that we leave this task for future investigations.

\paragraph{Outlook to higher dimensions}

\revision{We note that in particular the OBBT computations are much more 
expensive for higher dimensions since the number of linear programs that
have to be solved for each round of bound tightenings is proportional to
the number of grid cells in the local averaging; see also 
\cite{leyffer2025mccormick}. In order to handle the computational burden
in this case, one may, for example, use warm start techniques for the linear 
programs, parallelize their execution, and---in order to avoid memory exhaustion---replace the full set of McCormick inequalities by a cutting plane strategy.}
\Margin{R1 1.}

\section*{Acknowledgments}
Paul Manns acknowledges funding by Deutsche Forschungsgemeinschaft (DFG) under project no.\ 540198933.
Barbara Kaltenbacher acknowledges funding by the Austrian Science Fund (FWF) 
[10.55776/F100800]. 

\revisionown{The authors thank the reviewers for providing valuable comments and suggestions that have led to an improved version of the paper.}

\appendix 

\section{Regularization by Discretization -- Convergence of the Least Squares Method for Nonlinear Inverse Problems}
We consider an abstract inverse problem 
\begin{equation}\label{eq:Fxiy}
\mathbb{F}(\xi)=y, \quad \xi\in \mathcal{C}\subseteq \mathbb{X}
\end{equation}
with a forward operator $\mathbb{F}:\mathbb{X}\to \mathbb{Y}$, $\mathbb{X},\,\mathbb{Y}$ Banach spaces, and noisy data $y^\delta$ satisfying \eqref{eq:delta}
and the notation $\xi^\dagger\in \mathcal{C}$ for the exact solution to \eqref{eq:Fxiy}. Here $\mathcal{C}$ is a
closed convex subset of $\mathbb{X}$. Given a finite-dimensional subspace $\mathbb{X}_h$ of $\mathbb{X}$, we define a (discretization-)regularized
solution $\xi_h^\delta$ to \eqref{eq:Fxiy} by 
\begin{gather}\label{eq:leastsquares}
\xi_h^\delta \in \argmin
\left\{ \|\mathbb{F}(\xi)-y^\delta\|_{\mathbb{Y}} \colon
\xi\in \mathbb{X}_h\cap\mathcal{C}\right\}.
\end{gather} 
Existence of a minimizer follows by (weak/*) compactness of $\mathbb{X}_h\cap\mathcal{C}$ and (weak/*) lower semicontinuity of
the objective function under mild conditions on $\mathbb{F}$ and the spaces.

We derive a general convergence result by extending \cite[Theorem 3.7]{projBanach} straightforwardly by means of a
stability constant $\kappa_h$ that satisfies
\begin{equation}\label{eq:kappa}
\kappa_h \ge \sup_{\tilde{\xi}\not=\xi\in \mathbb{X}_h\cap\mathcal{C}}\frac{\|\tilde{\xi}-\xi\|_{\mathbb{X}}}{\|\mathbb{F}(\tilde{\xi})-\mathbb{F}(\xi)\|_{\mathbb{Y}}}.
\end{equation}
\begin{proposition}\label{prp:regdis}
	For $\xi_h^\delta$ defined by \eqref{eq:leastsquares}, the error estimate 
	\begin{gather}\label{eq:regdis_error_estimate}
	\|\xi_h^\delta-\xi^\dagger\|_{\mathbb{X}} 
	\leq \inf_{\xi_h\in \mathbb{X}_h}\Bigl(\|\xi_h-\xi^\dagger\|_{\mathbb{X}} 
\revisionown{+2 \kappa_h \|\mathbb{F}(\xi_h)-F(\xi^\dagger)\|_{\mathbb{Y}}}
\Bigr)
	+2 \kappa_h\,\delta
	\end{gather}
	holds with $\kappa_h$ defined by \eqref{eq:kappa}. 
	In particular, if the metric projection $P_{\mathbb{X}_h\cap\mathcal{C}}\xi^\dagger$ exists 
	for all $h\in(0,\bar{h}]$ for some $\bar{h}> 0$
	\footnote{where in place of the metric projection, we might use an arbitrary mapping $P_{\mathbb{X}_h\cap\mathcal{C}}:\mathbb{X}\to \mathbb{X}_h\cap\mathcal{C}$, which is anyway only evaluated at $\xi^\dagger$}	
	and tends to $\xi^\dagger$ fast enough so that  
	\begin{equation}\label{eq:conv-cond}
	\kappa_h \|\mathbb{F}(P_{\mathbb{X}_h\cap\mathcal{C}}\xi^\dagger)-\mathbb{F}(\xi^\dagger)\|_{\mathbb{Y}}\to0\text{ as }h\to0,
	\end{equation}
	then a regularization parameter choice such that 
	\[
	h(\delta)\to0 \text{ and } \kappa_{h(\delta)}\,\delta\to0 \text{ as }\delta\to0
	\]
	yields convergence $\|\xi_{h(\delta)}^\delta-\xi^\dagger\|_{\mathbb{X}}\to0$ as $\delta\to0$.
\end{proposition}
\begin{proof}
Given any $\xi_h\in \mathbb{X}_h\cap\mathcal{C}$, minimality yields
\[ \|\mathbb{F}(\xi_h^\delta)-y^\delta\|_{\mathbb{Y}} \leq \|\mathbb{F}(\xi_h)-y^\delta\|_{\mathbb{Y}} \]
and therefore
\begin{equation}\label{eq:xihdelta_min}
\|\mathbb{F}(\xi_h^\delta)-\mathbb{F}(\xi_h)\|_{\mathbb{Y}} = \|\mathbb{F}(\xi_h^\delta)-y^\delta-(\mathbb{F}(\xi_h)-y^\delta)\|_{\mathbb{Y}}\leq 2\|\mathbb{F}(\xi_h)-y^\delta\|_{\mathbb{Y}}.
\end{equation}
Using $\mathbb{F}(\xi^\dagger)=y$, the stability constant $\kappa_h$ defined in \eqref{eq:kappa}, \eqref{eq:delta},
and \eqref{eq:xihdelta_min}, we obtain
\begin{equation}\label{eq:xihdelta_xidagger}
\begin{aligned}
\|\xi_h^\delta-\xi^\dagger\|_{\mathbb{X}} 
&\leq \|\xi_h-\xi^\dagger\|_{\mathbb{X}} +\|\xi_h^\delta-\xi_h\|_{\mathbb{X}} \\
&\leq \|\xi_h-\xi^\dagger\|_{\mathbb{X}} + \kappa_h \|\mathbb{F}(\xi_h^\delta) - \mathbb{F}(\xi_h)\|_{\mathbb{Y}}
\\
&\leq \|\xi_h-\xi^\dagger\|_{\mathbb{X}} +2 \kappa_h (\|\mathbb{F}(\xi_h)-y^\delta\|_{\mathbb{Y}}+\delta).
\end{aligned}
\end{equation}
Using minimality and the stability assumption \eqref{eq:conv-cond} yields the second claim.
\end{proof}
\begin{remark}
Note that the formulation of the first result in this proposition is similar in spirit to Strang's first lemma.
For the latter convergence result it suffices to consider a slightly more specifically and a less restrictive
defined stability constant
\begin{equation}\label{eq:kappa_xidagger}
\kappa_h \ge \sup_{\xi\in \mathbb{X}_h\cap\mathcal{C}\setminus\{P_{\mathbb{X}_h\cap\mathcal{C}}\xi^\dagger\}}\frac{\|\xi -P_{\mathbb{X}_h\cap\mathcal{C}}\xi^\dagger\|_{\mathbb{X}}}{\|\mathbb{F}(\xi)-\mathbb{F}(P_{\mathbb{X}_h\cap\mathcal{C}}\xi^\dagger)\|_{\mathbb{Y}}}
\end{equation}
in place of \eqref{eq:kappa} by means of the choice $\xi_h = P_{\mathbb{X}_h \cap \mathcal{C}} \xi^\dagger$ in the argument.
This allows to make use of additional regularity of $\xi^\dagger$ if need be and in particular complies with our
stability constant assumption \eqref{eq:kappa_for_inv} for our specific problem setting above.
\end{remark}

\Cref{prp:regdis} shows that---as is well known for linear inverse problems---regularization by discretization \textit{in pre-image space}
only works if the approximation error tends to zero fast enough to counterbalance the growth of the stability constant.
The result can be extended to a setting with perturbed evaluation $\mathbb{F}_\tau$ of the forward operator and approximate optimization;
cf.\ \cite{NeubauerScherzer:1990}; as follows. To this end, let $\hat{\xi}_{\delta\,h\,\tau}$ be an approximate minimizer of
the perturbed problem
\begin{gather}\label{eq:leastsquares_tau}
\begin{aligned}
\min_{\xi}\ &\|\mathbb{F}_\tau(\xi)-y^\delta\|_{\mathbb{Y}}\\
\st\ \ & \xi\in \mathbb{X}_h\cap\mathcal{C},
\end{aligned}
\end{gather} 
that is,
\begin{equation}\label{eq:xihat_delta_tau}
\hat{\xi}_{\delta\,h\,\tau} \in \mathbb{X}_h\cap\mathcal{C}
\quad\text{and}\quad
\|\mathbb{F}_\tau(\hat{\xi}_{\delta\,h\,\tau})-y^\delta\|_{\mathbb{Y}} \leq \inf_{\xi\in \mathbb{X}_h\cap\mathcal{C}}\|\mathbb{F}_\tau(\xi)-y^\delta\|_{\mathbb{Y}}+\varepsilon.
\end{equation}
In this situation, we define
\begin{equation}\label{eqn:eta_tau_h}
\eta_{h\,\tau} \coloneqq \sup_{\xi_h\in \mathbb{X}_h\cap\mathcal{C}} \|\mathbb{F}_\tau(\xi_h)-\mathbb{F}(\xi_h)\|_{\mathbb{Y}}
\end{equation}
and obtain an alternative estimate.
\begin{proposition}\label{prp:regdis_tau}
Let $\hat{\xi}_{\delta\,h\,\tau}$ satisfy \eqref{eq:xihat_delta_tau}. Then the error estimate
\[
\|\hat{\xi}_{\delta\,h\,\tau}-\xi^\dagger\|_{\mathbb{X}} 
\leq \|P_{\mathbb{X}_h\cap\mathcal{C}}\xi^\dagger-\xi^\dagger\|_{\mathbb{X}} 
+\kappa_{h} \bigl(2\|\mathbb{F}(P_{\mathbb{X}_h\cap\mathcal{C}}\xi^\dagger)-\mathbb{F}(\xi^\dagger)\|_{\mathbb{Y}}+2\delta+2\eta_{h\,\tau}+\varepsilon\bigr)
\] 
holds with $\kappa_h$ satisfying \eqref{eq:kappa_xidagger} or \eqref{eq:kappa} and $\eta_{h\,\tau}$ defined by \eqref{eqn:eta_tau_h}.  
In particular, if for a proper choice $\tau=\tau(h)$    
\begin{equation}\label{eq:conv-cond_tau}
\kappa_{h} \bigl(\|\mathbb{F}(P_{\mathbb{X}_h\cap\mathcal{C}}\xi^\dagger)-\mathbb{F}(\xi^\dagger)\|_{\mathbb{Y}}+2\eta_{\tau(h)\,h}\bigr)\to0\text{ as }h\to0,
\end{equation}
then a precision and regularization parameter choice such that 
\[
\varepsilon=O(\delta), \quad 
h(\delta)\to0, \quad 
\revisionown{
\kappa_{h(\delta)}\,\delta\to0 \quad 
\kappa_{h(\delta)}\, \eta_{\tau(h(\delta))\,h(\delta)}\to0 \quad 
}
\text{ as }\delta\to0
\]
yields convergence $\|\hat{\xi}_{\delta\,h\,\tau}-\xi^\dagger\|_{\mathbb{X}}\to0$ as $\delta\to0$.
\end{proposition}
\begin{proof}
	We conclude from \eqref{eq:xihat_delta_tau} that for $\xi_h:=P_{\mathbb{X}_h\cap\mathcal{C}}\xi^\dagger$ we can estimate
	\[
	\begin{aligned}
	\|\mathbb{F}(\hat{\xi}_{\delta\,h\,\tau})-y^\delta\|_{\mathbb{Y}} 
	&\leq \|\mathbb{F}_\tau(\hat{\xi}_{\delta\,h\,\tau})-y^\delta\|_{\mathbb{Y}}
	+\|\mathbb{F}_\tau(\hat{\xi}_{\delta\,h\,\tau}) -\mathbb{F}(\hat{\xi}_{\delta\,h\,\tau})\|_{\mathbb{Y}}\\
	&\leq \|\mathbb{F}_\tau(\xi_h)-y^\delta\|_{\mathbb{Y}}+\varepsilon
	+\|\mathbb{F}_\tau(\hat{\xi}_{\delta\,h\,\tau}) -\mathbb{F}(\hat{\xi}_{\delta\,h\,\tau})\|_{\mathbb{Y}}\\
	&\leq \|\mathbb{F}(\xi_h)-y^\delta\|_{\mathbb{Y}}+\varepsilon
	+\|\mathbb{F}_\tau(\hat{\xi}_{\delta\,h\,\tau}) -\mathbb{F}(\hat{\xi}_{\delta\,h\,\tau})\|_{\mathbb{Y}}
	+\|\mathbb{F}_\tau(\xi_h) -\mathbb{F}(\xi_h)\|_{\mathbb{Y}}\\
	&\leq \|\mathbb{F}(\xi_h)-y^\delta\|_{\mathbb{Y}}+\varepsilon
	+2 \eta_{h\,\tau}.
	\end{aligned}
	\]
	Hence, by  the definition of $\eta_{h\,\tau}$,
	\[
	\begin{aligned}
	\|\mathbb{F}(\hat{\xi}_{\delta\,h\,\tau})-\mathbb{F}(\xi_h)\|_{\mathbb{Y}} 
	&= \|\mathbb{F}(\hat{\xi}_{\delta\,h\,\tau})-y^\delta-(\mathbb{F}(\xi_h)-y^\delta)\|_{\mathbb{Y}}
	\leq 2\|\mathbb{F}(\xi_h)-y^\delta\|_{\mathbb{Y}}+\varepsilon+2\eta_{h\,\tau}\\
	&\leq 2\|\mathbb{F}(\xi_h)-\mathbb{F}(\xi^\dagger)\|_{\mathbb{Y}}+2\delta+2\eta_{h\,\tau}+\varepsilon.
	\end{aligned}
	\]
	Combining this with 
	\[
	\|\hat{\xi}_{\delta\,h\,\tau}-\xi^\dagger\|_{\mathbb{X}} 
	\leq \|\xi_h-\xi^\dagger\|_{\mathbb{X}} +\kappa_h \|\mathbb{F}(\hat{\xi}_{\delta\,h\,\tau})-\mathbb{F}(\xi_h)\|_{\mathbb{Y}}
	\]
	yields the assertion.
\end{proof}
\end{document}